\documentclass[11pt]{amsart}
\usepackage{amssymb,amsmath,amsthm,amsfonts}
\usepackage {latexsym}
\usepackage{bbm}

\include{srctex}

\usepackage{marginnote}
\usepackage[letterpaper, top=1.2in, bottom=1.2in, outer=1.1in, inner=1.1in, heightrounded, marginparwidth=2.5cm, marginparsep=2mm]{geometry}

\allowdisplaybreaks
\newcommand{\nc}{\newcommand}
\nc{\les}{\lesssim}
\nc{\nit}{\noindent}
\nc{\nn}{\nonumber}
\nc{\D}{\partial}
\nc{\diff}[2]{\frac{d #1}{d #2}}
\nc{\diffn}[3]{\frac{d^{#3} #1}{d {#2}^{#3}}}
\nc{\pdiff}[2]{\frac{\partial #1}{\partial #2}}
\nc{\pdiffn}[3]{\frac{\partial^{#3} #1}{\partial{#2}^{#3}}}
\nc{\abs}[1] {\lvert #1 \rvert}
\nc{\cAc}{{\cal A}_c}
\nc{\cE}{{\cal E}}
\nc{\cF}{{\cal F}}
\nc{\cP}{{\cal P}}
\nc{\cV}{{\cal V}}
\nc{\cQ}{{\cal Q}}
\nc{\cGin}{{\cal G}_{\rm in}}
\nc{\cGout}{{\cal G}_{\rm out}}
\nc{\cO}{{\cal O}}
\nc{\Lav}{{\cal L}_{\rm av}}
\nc{\cL}{{\cal L}}
\nc{\cB}{{\cal B}}
\nc{\cZ}{{\cal Z}}
\nc{\cR}{{\cal R}}
\nc{\cT}{{\cal T}}
\nc{\cY}{{\cal Y}}
\nc{\cX}{{\cal X}}
\nc{\cXT}{{{\cal X}(T)}}
\nc{\cBT}{{{\cal B}(T)}}
\nc{\vD}{{\vec \mathcal{D}}}
\nc{\efield}{\mathcal{E}}
\nc{\vE}{{\vec \efield}}
\nc{\vB}{{\vec \mathcal{B}}}
\nc{\vH}{{\vec \mathcal{H}}}
\nc{\mR}{\mathcal{R}}
\nc{\mG}{\mathcal{G}}
\nc{\ty}{{\tilde y}}
\nc{\tu}{{\tilde u}}
\nc{\tV}{{\tilde V}}
\nc{\Pc}{{\bf P_c}}
\nc{\bx}{{\bf x}}
\nc{\bX}{{\bf X}}
\nc{\bXYZ}{{\bf XYZ}}
\nc{\bY}{{\bf Y}}
\nc{\bF}{{\bf F}}
\nc{\bS}{{\bf S}}
\nc{\dV}{{\delta V}}
\nc{\dE}{{\delta E}}
\nc{\TT}{{\Theta}}
\nc{\dPsi}{{\delta\Psi}}
\nc{\order}{{\cal O}}
\nc{\Rout}{R_{\rm out}}
\nc{\eplus}{e_+}
\nc{\eminus}{e_-}
\nc{\epm}{e_\pm}
\nc{\eps}{\varepsilon}
\nc{\vnabla}{{\vec\nabla}}
\nc{\G}{\Gamma}
\nc{\w}{\omega}
\nc{\mh}{h}
\nc{\mg}{g}
\nc{\vphi}{\varphi}
\nc{\tlambda}{\tilde\lambda}
\nc{\be}{\begin{equation}}
\nc{\ee}{\end{equation}}
\nc{\ba}{\begin{eqnarray}}
\nc{\ea}{\end{eqnarray}}

\nc{\g}{\gamma}
\nc{\ol}{\overline}

\newtheorem{theorem}{Theorem}[section]
\newtheorem{lemma}[theorem]{Lemma}
\newtheorem{prop}[theorem]{Proposition}
\newtheorem{corollary}[theorem]{Corollary}
\newtheorem{defin}[theorem]{Definition}
\newtheorem{rmk}[theorem]{Remark}
\newtheorem{asmp}[theorem]{Assumption}

\nc{\pT}{\partial_T}
\nc{\pz}{\partial_z}
\nc{\pt}{\partial_t}
\nc{\la}{\langle}
\nc{\ra}{\rangle}
\nc{\infint}{\int_{-\infty}^{\infty}}
\nc{\halfwidth}{6.5cm}
\nc{\figwidth}{10cm}
\newcommand{\f}{\frac}

\nc{\nlayers}{L} \nc{\nsectors}{M}
\nc{\indicator}{\mathbf{1}}
\nc{\Rhole}{R_{\rm hole}}
\nc{\Rring}{R_{\rm ring}}
\nc{\neff}{n_{\rm eff}}
\nc{\Frem}{F_{\rm rem}}
\nc{\R}{\mathbb R}
\nc{\C}{\mathbb C}
\nc{\Z}{\mathbb Z}
\nc{\N}{\mathbb N}
\nc{\DD}{\Delta}
\nc{\cD}{\mathcal D}
\nc{\lnorm}{\left\|}
\nc{\rnorm}{\right\|}
\nc{\rnormp}{\right\|_{\ell^{p,\eps}}}
\nc{\rar}{\rightarrow}
\begin{document}

\begin{abstract}

	We investigate dispersive estimates for the two dimensional Dirac  equation with a potential.  In particular, we show that the Dirac evolution satisfies a $t^{-1}$ decay rate as an operator from the Hardy space $H^1$ to $BMO$, the space of functions of bounded mean oscillation.  This estimate, along with the $L^2$ conservation law allows one to deduce a family of Strichartz estimates.  We classify the structure of threshold obstructions as being composed of s-wave resonances, p-wave resonances and eigenfunctions.  We show that, as in the case of the Schr\"odinger evolution, the presence of a threshold s-wave resonance does not destroy the $t^{-1}$ decay rate. As a consequence of our analysis we obtain a limiting absorption principle in the neighborhood of the threshold, and  show that there are only finitely many eigenvalues in the spectral gap.
	
\end{abstract}

\title[The Dirac Equation in Two Dimensions]{\textit{The Dirac equation in two dimensions: Dispersive estimates and classification of threshold obstructions}}

\author[M.~B. Erdo\smash{\u{g}}an, W.~R. Green]{M. Burak Erdo\smash{\u{g}}an and William~R. Green}
\thanks{The first author was partially supported by the NSF grant  DMS-1501041.}
\address{Department of Mathematics \\
University of Illinois \\
Urbana, IL 61801, U.S.A.}
\email{berdogan@math.uiuc.edu}
\address{Department of Mathematics\\
Rose-Hulman Institute of Technology \\
Terre Haute, IN 47803, U.S.A.}
\email{green@rose-hulman.edu}

\maketitle
\section{Introduction}

We consider the linear Dirac equation with a potential:
\begin{align}\label{eqn:Dirac}
	i\partial_t \psi(x,t)=(D_m+V(x))\psi(x,t), \qquad
	\psi(x,0)=\psi_0(x).
\end{align}
Here  the spatial
variable $x\in \mathbb R^n$, and $\psi(x,t) \in \mathbb C^{2^{n-1}}$.  The free Dirac operator
$D_m$ is defined by
\begin{align}\label{eqn:Dmdef}
	D_m=-i\alpha \cdot \nabla +m\beta =-i \sum_{k=1}^{n}\alpha_k \partial_{k}+m\beta
\end{align}
where $m>0$ is a constant, and 
the $n\times n$ Hermitian matrices $\alpha_0:=\beta$ and $\alpha_j$ satisfy
\be \label{eqn:anticomm}
		\alpha_j \alpha_k+\alpha_k\alpha_j =2\delta_{jk}
		\mathbbm 1_{\mathbb C^{2^{n-1}}}, \,\,\,\,\,\,\,
		  j,k \in\{0, 1,2,\dots, n\}.
\ee
For concreteness, in two dimensions 
we use
\begin{align}
	\beta=\left(\begin{array}{cc} 1& 0\\ 0 & -1
	\end{array}\right), \qquad \alpha_1=\left(\begin{array}{cc} 0 & 1\\ 1 & 0
	\end{array}\right), \qquad
	\alpha_2=\left(\begin{array}{cc} 0 & -i\\ i & 0
	\end{array}\right).
\end{align}
Dirac arrived
at these equations to describe the evolution of an electron moving at relativistic speeds, thus the Dirac equation is a way to connect the physical theories of quantum mechanics and relativity, see,  e.g., \cite{Thaller}.   The Dirac equation can be derived by applying quantum-mechanical notions of energy $E=i\hbar \partial_t$ and momentum $p=-i\hbar \nabla$ to the relativistic relationship between energy, momentum and mass, $E=\sqrt{c^2p^2+m^2c^4}$.  One arrives at the square root of a Klein-Gordon equation,
$$
	i\hbar \partial_t \psi(x,t)=\sqrt{-c^2\hbar^2 \Delta 
	+m^2c^4}\, \psi(x,t)
$$
Here $\hbar$ is Planck's constant and $c$ is the speed of light.  In our mathematical analysis, we rescale all constants to be one.  Dirac's linearization of the above equation led to the free Dirac equation, a system of coupled hyperbolic equations, \eqref{eqn:Dirac} with $V=0$.  Dirac's linearization allows one to account for the spin of quantum particles, as well as providing a way to incorporate external electro-magnetic fields in a manner compatible with the relativistic theory where the Klein-Gordon model cannot.  Further details can be found in \cite{Thaller}.

The following identity,\footnote{Here and throughout the paper, scalar operators such as  $-\Delta+m^2-\lambda^2$ are understood as $(-\Delta+m^2-\lambda^2)\mathbbm 1_{\mathbb C^{2^{n-1}}}$.}  which follows from   \eqref{eqn:anticomm},
\be  \label{dirac_schro_free}
	(D_m-\lambda \mathbbm 1)(D_m+\lambda \mathbbm 1) =(-i\alpha\cdot \nabla +m\beta -\lambda \mathbbm 1)
	(-i\alpha\cdot \nabla+m\beta+\lambda \mathbbm 1)   =(-\Delta+m^2-\lambda^2) 
\ee
allows us to formally define the free Dirac resolvent
operator $\mathcal R_0(\lambda)=(D_m-\lambda)^{-1}$ in terms of the
free resolvent $R_0(\lambda)=(-\Delta-\lambda)^{-1}$ of  the Schr\"odinger operator for $\lambda$ in the resolvent set:
\begin{align}\label{eqn:resolvdef}
	\mathcal R_0(\lambda)=(D_m+\lambda) R_0(\lambda^2-m^2).
\end{align}

We note that
$$
	\sigma(D_m)=\sigma_{ess}(D_m)=(-\infty,-m]\cup [m,\infty),
$$
and for suitable potential functions $V$, one has
$\sigma_{ess}(H)=\sigma_{ess}(D_m)$ with $H=D_m+V$.
This is satisfied for large classes of potentials, for 
examples if $V(x)\to 0$ as $|x|\to \infty$, see
\cite[Theorem 4.7]{Thaller}, or \cite{MY}. For the class of potentials we consider in this paper,  Georgescu and Mantoiu \cite[Theorem 1.4]{GM}  proved that  there is no singular  continuous spectrum of $H$, also see \cite{Yam}. Furthermore, the set of eigenvalues is a discrete subset of $\R\backslash\{m,-m\}$, and each eigenvalue is of finite multiplicity,  see \cite{GM} and \cite{Co1}.  
It is possible that eigenvalues accumulate at $\pm m$ even for bounded and decaying potentials, see \cite{Thaller}. However, the resolvent expansions we obtain in Section~\ref{sec:m exp} rule that out, see Remark~\ref{rmk:LAP}. In particular there are finitely many eigenvalues in the spectral gap $(-m,m)$; also see  \cite{Mur}, \cite{kopy}, \cite{Kurb}, and \cite{Co}. 
In three dimensions, for suitably decaying potentials, it is known that  there are no embedded eigenvalues in the essential spectrum, except
possibly at the thresholds $\lambda=\pm m$, \cite{Thomp}; also see also \cite{R,BG1,Vog,GM}. Although the method in \cite{Thomp} appears to be applicable in general dimensions and for more general operators, this has not been pursued anywhere. The method of Roze in \cite{R} is based on squaring the equation and using analogous results for Schrodinger type operators. This is applicable in two dimensions, however it only applies for potentials of the form $q(x)I$.

For our analysis, we make the following assumptions on the potential $V$ and the spectrum of $H=D_m+V$.

\begin{asmp}\label{assumption}
i) The matrix-valued potential function $V(x)$ is self-adjoint and its entries satisfy the decay estimate $|V_{ij}(x)|\les \la x\ra^{-\gamma}$ for some $\gamma>3$. \\
ii) There are no eigenvalues in $(-\infty,-m)\cup (m,\infty)$.\\
iii) A limiting absorption principle  for the perturbed resolvent operator of the form:
\be\label{eqn:lap}
			\sup_{|\lambda|>\lambda_0}\|\partial_\lambda^k \mR_V^\pm (\lambda)\|_{L^{2,\sigma}\to L^{2,-\sigma}} \les 1,
			\qquad \sigma > \f12 +k, \,\,\,\,k=0,1, 
\ee  
		holds for any $\lambda_0>m$. 
\end{asmp}
We note that the second and third assumptions are used only in the high energy analysis of the operator.  The third assumption above  requires some discussion.   Note that in contrast with the 
Schr\"odinger resolvent, the resolvent for the free Dirac operator does not decay as $\lambda\to \infty$.  Therefore, Agmon's bootstrapping  argument  \cite{agmon} produces uniform bounds in $\lambda$ only for compact subsets of the purely absolutely continuous spectrum, see e.g. \cite{Yam}, \cite{GM}. In particular, under the first two assumptions, the results of \cite{GM} imply that
$$
\sup_{\lambda_1>|\lambda|>\lambda_0}\|\partial_\lambda^k \mR_V^\pm (\lambda)\|_{L^{2,\sigma}\to L^{2,-\sigma}} \les 1, \qquad \sigma > \f12 +k, \,\,\,\,k=0,1, 
$$
for any $\lambda_1> \lambda_0>m$, with a bound depending on both $\lambda_1$ and $\lambda_0$. This situation is quite similar to the case of magnetic Schr\"odinger equation, and it is likely that one can obtain \eqref{eqn:lap} under the first two assumptions   using the method in \cite{EGS}. This will be pursued elsewhere. 
 
The limiting absorption principle  is much  better studied in the three dimensional case, see for example \cite{Bouss1,DF,BG}.  The results of D'Ancona and Fanelli, \cite{DF},  requires the potential to be small.  Boussaid and Gol\'enia, \cite{BG}, established a limiting absorption principle near the thresholds $\lambda=\pm m$. Finally, Georgescu and Mantoiu \cite{GM} obtained a limiting absorption principle in general dimensions but it is only uniform on compact subsets of the purely absolutely continuous spectrum. 

As our   time decay analysis requires only the decay assumption but  no smoothness  or smallness in any particular norm, we chose to leave this as an overarching assumption.  In particular, we need only use this assumption in our high energy analysis in Section~\ref{sec:high}.  A viable limiting absorption principle of this form may be attained for a non-trivial class of potentials, following the work of Boussaid \cite{Bouss1} in three dimensions provided the potential is $C^\infty$ and satisfies the decay estimates $|\partial_x^k V(x)|\les \la x \ra^{-5-k-}$. This approach makes use of the minimal escape velocity estimates of Hunziker, Sigal and Soffer, \cite{HSS}, to establish  time decay on polynomially weighted $L^2$ spaces.  Then, one can use iterated resolvent identities to establish \eqref{eqn:lap}, see \cite{kopy}.

To state our main result we introduce some notation. Throughout the paper   $a-:=a-\epsilon$ for an arbitrarily small, but fixed, $\epsilon>0$.  Similarly, $a+:=a+\epsilon$. Let  $P_{ac}$ be the spectral projection onto the absolutely continuous spectrum. For the definition of the threshold regularity and resonances see Definition~\ref{resondef} below.  
\begin{theorem}\label{thm:main}
 Suppose Assumption~\ref{assumption} holds, and the threshold energies, $\pm m$, are regular or there are resonances of the first kind at $\pm m$.  
	Then, for any $0\leq \alpha \leq 1$, we have the dispersive bounds for the Dirac evolution\footnote{Thoroughout the paper $\|\cdot\|_{BMO} $ denotes $\|\cdot\|_{BMO\times BMO(\R^2)}$. We similarly use the notation $\|\cdot\|_{H^1}$ and $\|\cdot\|_{L^p}$. See \cite{FS} for the definition of the spaces $H^1$ and $BMO$.} 
	\begin{align}\label{eqn:mainthm}
		\|e^{-itH} P_{ac}(H)\langle H\rangle^{-2-\frac{3}{2}\alpha -} f\|_{BMO } \les \frac{1}{\la t\ra^{\alpha}}
		\|f\|_{H^1 }.
	\end{align}	
 \end{theorem}
The free Dirac evolution has   threshold s-wave resonances, thus Theorem~\ref{thm:main} holds under the natural conditions on the edge of the spectrum.

As usual, we study the dispersive bounds on the evolution by expressing  $e^{-itH}P_{ac}(H)$   via the Stone's formula:
\begin{align}\label{Stone}
	e^{-it H}P_{ac}(H)f(x)=\frac{1}{2\pi} \int_{\sigma_{ess}(H)} e^{-it\lambda} [\mR_V^+-\mR_V^-](\lambda) f(x)\,d\lambda.
\end{align}

Due to the differing behavior of the resolvents
$\mR_V^\pm(\lambda)$ in a neighborhood of the threshold and away from the threshold on the positive half of the spectrum, our analysis proceeds in two cases.  We first consider the low-energy contribution, when $\lambda$ is in a sufficiently small neighborhood of the threshold
$\lambda=m$.  A similar analysis can be done on the negative portion of the spectrum.

\begin{theorem}\label{thm:mainlow energy}

	Under Assumption~\ref{assumption} part i), with $\chi$ a smooth cut-off to a sufficiently small neighborhood of the threshold energy
	$\lambda=m$. If $\lambda=m$ is regular or if there is 
	a resonance of the first kind, then the following dispersive bound holds.
	\begin{align*}
		\| e^{-itH}P_{ac}(H)\chi(H) f\|_{BMO }\les
		\frac{1}{\la t\ra} \|f\|_{H^1 }.
	\end{align*}  
\end{theorem}

When $\lambda$ is away from the threshold, we consider the evolution concentrated on dyadic frequencies to prove the following.

\begin{theorem}\label{thm:high energy}
	Under Assumption~\ref{assumption},
	with $\chi_j$ a smooth cut-off to $\lambda \approx 2^j$, $j\geq 0$, 
	we have the bound
	\begin{align*}
		\| e^{-itH}P_{ac}(H)\chi_j(H)f\|_{BMO }\les \min(2^{2j}, 2^{\f72 j}|t|^{-1})
		\|f\|_{H^1 }.
	\end{align*}

\end{theorem}

Therefore, we obtain 
$$
	\| e^{-itH}\la H\ra^{-\frac72-} P_{ac}(H) f\|_{BMO }\les |t|^{-1} 
		\|f\|_{H^1 }.
$$
Interpolating this bound with the $L^2$ conservation (see, e.g.  \cite{BS,Hanks}) one obtains the $L^p\to L^{p^\prime}$ bound:
$$
	\| e^{-itH}\la H\ra^{ \frac72-\frac7p-}  P_{ac}(H) f\|_{L^{p^\prime}}\les |t|^{1-\frac2p} \|f\|_{L^p},\,\,\,\,1 <  p \leq 2.  
$$
As in the seminal work of Ginibre and Velo \cite{GV}, this yields the following Strichartz estimates:
\begin{corollary}
 Under the conditions of  Theorem~\ref{thm:main}, we have   
$$
 \| e^{-itH}\la H\ra^{ \frac7{2r}-\frac74-} P_{ac}(H) f\|_{L^{q}_tL^r_x}\les   \|f\|_{L^2},\,\,\,\,\frac1q+\frac1r=\frac12, \,\,\,\,\,2\leq r<\infty.
$$
\end{corollary}

Since the time-decay we obtain is the same as that for the Schr\"odinger equation in two dimensions, the range of admissible exponents for the Dirac evolution mirrors that of Schr\"odinger evolution.   

The mathematical analysis of Dirac operators is less well-studied than the related Schr\"odinger, wave and Klein-Gordon equations.  The paper \cite{DF} by   D'Ancona and Fanelli  seems to be the first  to analyze the time-decay for a perturbed Dirac evolution in a pointwise sense.  This analysis in the three-dimensional case considered the massless Dirac equation ($m=0$) and related wave equations with small electromagnetic potentials.  Dispersive and Strichartz estimates for the free Dirac equation in three dimensions\footnote{During the review period for this article, the authors and Toprak studied the analagous dispersive estimates for three dimensional Dirac equations with threshold obstructions, \cite{EGT}.} were proven by Escobedo and Vega, \cite{EV}, to analyze a semi-linear Dirac equation.  Boussaid, \cite{Bouss1} proved dispersive estimates in Besov spaces and weighted $L^2$ spaces for the three-dimensional Dirac equation with mass $m>0$, with an aim towards studying the stable manifold around `particle-like solutions' for a class of non-linear Dirac equations.   

Further study of the Dirac operator in the sense of smoothing and Strichartz estimates has been performed by a variety of authors, see for example \cite{BDF,C,CS}.  In the two-dimensional case,   the evolution on weighted $L^2$ spaces was studied in \cite{kopy}, which had roots in the work of Murata, \cite{Mur}.   Frequency-localized endpoint Strichartz estimates for the free Dirac equation are obtained in two spatial dimensions in \cite{BH}, which are used to study the cubic non-linear Dirac equation. Dispersive estimates for one-dimensional Dirac equation was considered in \cite{CTS}. 

Our approach relies on a detailed analysis of the Dirac resolvent operators.  We follow the strategy employed to analyze the two-dimensional Schr\"odinger equation set out in \cite{Sc2} and in our earlier work \cite{eg2,eg3,eg4}, also see \cite{ebru}.  We note that extending these results is  non-trivial even for the wave equation, see \cite{Gwave,beceanu}.  In contrast to the Schr\"odinger and wave equation, we present our estimates from $H^1$ to $BMO$ instead of as operators from $L^1\to L^\infty$.  The use of such spaces are necessitated by technical issues which we discuss in Section~\ref{sec:low}, however they still serve as interpolation spaces in the same way as $L^1$ and $L^\infty$. Dispersive estimates in the setting of $H^1$ and $BMO$ spaces were established in \cite{DF1,beceanu}.

In addition to proving time decay estimates for the Dirac evolution, we provide a full classification of the obstructions that can occur at the threshold of the essential spectrum at $\lambda =\pm m$.  In two dimensions, there is a rich structure of resonances and eigenfunctions that can occur, which we classify.  This classification is inspired by the previous work on Schr\"odinger operators   \cite{JN,BDG,eg2}. For the classification of threshold obstructions in three dimensions and their effect on the time decay of the Dirac evolution see \cite{EGT}.

The paper is organized as follows.  We first develop expansions for the free Dirac resolvent around the threshold energy $\lambda=m$ in Section~\ref{sec:exps}.  These bounds allow us to prove the natural time decay bounds for the free Dirac evolution as an operator between Besov spaces  in Section~\ref{sec:free} as well as to establish expansions for the perturbed resolvent near the threshold in Section~\ref{sec:m exp}.  These expansions then allow us to prove bounds on the low-energy portion of the evolution in Section~\ref{sec:low}. We prove bounds on the high-energy portion of the evolution in Section~\ref{sec:high}.  Finally, we  classify the threshold resonances and eigenfunctions in Section~\ref{sec:resonanceclass}.

\section{Free resolvent expansions around the threshold energy $m$}\label{sec:exps}
    
In this section we study the behavior of the free  Dirac resolvent by using the properties of   free Schr\"odinger resolvent operator $R_0(z)=(-\Delta-z)^{-1}$.  Formally, the free resolvent $$R_0^\pm(z^2)=\lim_{\epsilon\to 0^+}(-\Delta -(z^2\pm i\epsilon))^{-1}$$ and the perturbed resolvent operators $$R_V^\pm(z^2)=\lim_{\epsilon\to 0^+}(-\Delta+V -(z^2\pm i\epsilon))^{-1}$$
are well-defined as an operator between weighted $L^2(\R^2)$ spaces, see \cite{agmon}.  

Here we review some estimates (see e.g. \cite{Sc2,eg2,eg3}) for $R_0^\pm(z^2)$ needed to study the Dirac evolution.  To best utilize these expansions, we employ the notation
$$
f(z)=\widetilde O(g(z))
$$
to denote
$$
\frac{d^j}{dz^j} f = O\big(\frac{d^j}{dz^j} g\big),\,\,\,\,\,j=0,1,2,3,...
$$
The notation refers to derivatives with respect to the spectral variable $z$, or $|x-y|$ in the expansions for
the integral kernel of the free resolvent operator, which is a function of the variable $\rho=z|x-y|$.
If the derivative bounds hold only for the first $k$ derivatives we  write $f=\widetilde O_k (g)$.  In addition, if we write $f=\widetilde O_k(1)$, we mean that differentiation up to order $k$ is comparable to division by $z$ and/or $|x-y|$.
This notation applies to operators as well as
scalar functions; the meaning should be clear from the
context.

\begin{lemma} \label{lem:schro_resol} Let $0<z\ll 1$. For $z|x-y|<1$, we have the expansions
	\begin{multline}\label{eq:freschroresolv}
		R_0^{\pm}(z^2) =g^\pm(z)+G_0
		+  \widetilde 
		O_2(z^2|x-y|^2\log(z|x-y|))\\
		 =g^\pm(z)+G_0
		+g_1^\pm(z) G_1+z^2 G_2+ \widetilde 
		O_2(z^4|x-y|^4\log(z|x-y|)),
	\end{multline}
 where
\begin{align}
	g^\pm(z)&=-\frac1{2\pi} \big(\log(z/2)+\gamma\big)\pm\frac{i}4 \label{g def}\\
	g_1^\pm(z)&=-\frac{z^2}4g^\pm(z)-\frac{z^2}{8\pi}  \label{g1 def}
\end{align}
\begin{align}
	G_0f(x)&=-\frac{1}{2\pi}\int_{\R^2} \log|x-y|f(y)\,dy, \label{G0 def}\\
	G_1f(x)&=\int_{\R^2} |x-y|^2f(y)\, dy,\label{G1 def}\\
	G_2f(x)&= \frac1{8\pi}\int_{\R^2} |x-y|^2\log|x-y|f(x)\, dy.\label{G2 def}
\end{align}
For $z|x-y|>1$, we have 
\be \label{R0large}
	R_0^\pm(z^2)(x,y)=   e^{\pm iz|x-y|} \omega_\pm(z|x-y|), \,\,\,\,\,\,
	|\omega_{\pm}^{(j)}(y)|\lesssim (1+|y|)^{-\frac{1}{2}-j},\,\,\,j=0,1,2,\ldots.
\ee
\end{lemma}
We develop expansions on the positive portion of the
spectrum, $[m,\infty)$.   The negative branch, $(-\infty,-m]$, can be studied analogously with only minor changes, see Remark~\ref{rmk:m swave} below. We write $\lambda=\sqrt{m^2+z^2}$ with 
$0<z\ll 1$. Using  \eqref{eqn:resolvdef} we have
\begin{multline}\label{eq:dr1}
	\mR_0^\pm(\lambda)=\left[-i\alpha \cdot \nabla + m \beta +\sqrt{m^2+z^2} I\right]R_0^\pm(z^2) = \\  
	\left[-i\alpha \cdot \nabla + m (\beta +I)+\frac{z^2}{2m} I +\widetilde O(z^4) I\right]R_0^\pm(z^2).
\end{multline}
We now employ the following notational conventions. The operators  $M_{11}$ and $M_{22}$
are defined to be matrix-valued operators with kernels 
\begin{align*}
	M_{11}(x,y)=\left(\begin{array}{cc}
		1 & 0 \\ 0 & 0
	\end{array}\right), \qquad M_{22}(x,y)=\left(\begin{array}{cc}
		0 & 0 \\ 0 & 1
	\end{array}\right).
\end{align*}
We also define the projection operators $I_1, I_2$ by
\begin{align*}
	I_1 \left(\begin{array}{c}
		a \\ b\end{array}\right)=\left(\begin{array}{c}
		a \\0\end{array}\right),\qquad
	I_2 \left(\begin{array}{c}
		a \\ b\end{array}\right)=\left(\begin{array}{c}
		0 \\ b\end{array}\right).	
\end{align*}
Using \eqref{eq:freschroresolv} and \eqref{eq:dr1}, we have (for $z|x-y|<1$, $0<z\ll 1$, $\lambda=\sqrt{z^2+m^2}$)
\begin{multline} \label{r0temp}  
	\mR_0^\pm(\lambda)=	 
	\left[-i\alpha \cdot \nabla + 2m I_1 +\frac{z^2}{2m} I +\widetilde O(z^4) I\right]    \\                                                             
	\left[g^\pm(z)+G_0+g_1^\pm(z) G_1 +z^2 G_2+ \widetilde O_2(z^4|x-y|^4\log(z|x-y|))\right].
\end{multline}
We define
\begin{align} \label{def:mG0}
	\mathcal G_0&=-i\alpha \cdot \nabla G_0 
	+2mG_0 I_{1}\\
 \label{def:mG1}	\mathcal G_1&=-i\alpha \cdot \nabla G_1
	+ 2mG_1 I_{1}-\frac2mM_{11}-\frac2mM_{22}\\
 \label{def:mG2}	\mathcal G_2&=-i\alpha \cdot \nabla G_2+2m G_2 I_1+\frac1{2m}G_0 -\frac1{4\pi m}M_{11}-\frac1{4\pi m}M_{22}.
\end{align}
We further define the function $\log^-(y):=-\log (y) \chi_{\{0<y<1\}}$.
Using this notation, the expansion \eqref{r0temp}  can be written as  
\be \label{eq:R0exp}
	  \mR_0^\pm(\lambda) =	  2m g^\pm (z) M_{11}+\mathcal G_0+  
	 \widetilde O_2\big(z^{k}(|x-y|^{k}+\log^-|x-y|)\big),\,\,\,0<k<2,
\ee
or as
\begin{multline}	 
	\label{eq:R0exp2} \mR_0^\pm(\lambda) = 2m g^\pm (z) M_{11}+\mathcal G_0+ g_1^\pm (z) \mathcal G_1   + z^2 \mathcal G_2 \\
	+\widetilde O_2\big(z^{\ell}(|x-y|^{\ell}+\log^-|x-y|)\big),\,\,\,2<\ell<4.
\end{multline}
To obtain these formulas, we write using \eqref{g1 def} that 
$$
\frac{z^2}{2m}g^\pm(z)=\frac{z^2}{2m}g^\pm(z)[M_{11}+M_{22}]= -\frac{z^2}{4\pi m}[M_{11}+M_{22}]-\frac2m g_1^\pm(z) [M_{11}+M_{22}].
$$
In this expansion we chose to group terms by their size with respect
to the spectral variable $z$  rather than by operator.

Combining the expansions \eqref{eq:R0exp} and \eqref{eq:R0exp2} with the high energy expansion \eqref{R0large}, we obtain 
\begin{lemma}\label{lem:R0exp}	
	We have the following expansion for the kernel of the free resolvent, $\lambda=\sqrt{m^2+z^2},$ $0<z\ll 1$
	\begin{align*}
		\mR_0^{\pm}(\lambda)(x,y) =2m  g^\pm (z) M_{11}+\mathcal G_0(x,y)
		+E_0^{\pm}(z)(x,y),
	\end{align*}
	where $E_0^{\pm}$ satisfies the bounds 
	$$
		|E_0^{\pm}|\les z^{k }(|x-y|^{k }+\log^-|x-y|), \,\,
		|\partial_z E_0^{\pm}|\les z^{k-1}(|x-y|^{k }+\log^-|x-y|),  
	$$
	for any $\frac12\leq k<2$.
	Furthermore, we have
	$$ 
	E_0^{\pm}(z)(x,y)=g_1^\pm (z) \mathcal G_1   + z^2 \mathcal G_2+E_1^\pm(z)(x,y),
	$$
	where 
	$$
	|E_1^{\pm}|\les z^{\ell }(|x-y|^{\ell}+\log^-|x-y|), \,\,
		|\partial_z E_1^{\pm}|\les z^{\ell-1}(|x-y|^{\ell}+\log^-|x-y|), 
	$$
	for any $2<\ell<4$.
\end{lemma}
\begin{proof}
For $z|x-y|<1$, we already obtained the required bound in \eqref{eq:R0exp}.

 For $z |x-y|\gtrsim 1$, using   \eqref{R0large}, 
	we have (with $\lambda=\sqrt{m^2+z^2}$)
	\begin{multline*} 
	\mR_0^\pm(\lambda)=[-i\alpha \cdot \nabla + m \beta +\sqrt{m^2+z^2} I] R_0^\pm(z^2)   \\  
	=[-i\alpha \cdot \nabla + m \beta +\sqrt{m^2+z^2} I] [e^{\pm iz|x-y|} \omega_\pm(z|x-y|)].
\end{multline*}
Therefore, for $z|x-y|\gtrsim 1$ and $0<z\ll1$  (in this case $|x-y|\gtrsim 1$), we have
\begin{align*}
|\partial_z^j \mR_0^\pm(\lambda)(x,y)|\les z^{-1/2}|x-y|^{j-1/2}.
\end{align*}
Also using
$$
		E_0^\pm(z)  =  \mR_0^\pm(\lambda)  - 2m g^\pm (z) M_{11} -  \mG_0,
		 $$
we have for $j=0,1,$
	\begin{align*}
		\left|\partial_z^j E_0^\pm(z)(x,y)\chi_{\{z|x-y|>1\}}\right| \lesssim   z^{-1/2}|x-y|^{j-1/2} + z^{-j}(z|x-y|)^{0+}.
	\end{align*}
Now, note that for any $\frac12\leq k<2$, and for $j=0,1$, we have
	\begin{multline*}
		|\partial_z^j E_0^{\pm}(z)(x,y)|  \lesssim      \big[z^{k-j}(|x-y|^{k}+\log^-|x-y|) \chi_{\{z|x-y|<1\}}\\
		+ ( z^{-1/2}|x-y|^{j-1/2}+ z^{-j}(z|x-y|)^{0+})\chi_{\{z|x-y|>1\}} \big] \\ \lesssim  z^{k-j}(|x-y|^{k}+\log^-|x-y|).
		\end{multline*}
		The proof for $E_1^\pm$ is similar. We already obtained the required bound in the case $z|x-y|<1$ in \eqref{eq:R0exp2}. 
		For $z|x-y|>1$, using the high energy estimate above, and
		\begin{align*} 
		 E_1^{\pm}   = \mR_0^\pm(\lambda)   -  2m  g^\pm (z) M_{11}-  \mG_0 -g_1^\pm (z) \mathcal G_1  - z^2  \mathcal G_2,
	\end{align*}
	we obtain (note that $|x-y|\gtrsim 1$)
$$
		\left|\partial_z^j E_1^{\pm}(z)(x,y)\chi_{\{z|x-y|>1\}}\right|    \lesssim  z^{-1/2}|x-y|^{j-1/2} + z^{-j}(z|x-y|)^{2+}    \les z^{ -j }(z|x-y|)^{2+}.
$$
	Hence, for any $2<\ell<4$ and $j=0,1,$ we have
	 \begin{multline*}	
		\left|\partial_z^j E_1^{\pm}(z)(x,y)\right|  \lesssim    \big[  z^{\ell-j } (|x-y|^{\ell}+ \log^-|x-y|)  \chi_{\{z|x-y|<1\}}
		+ z^{ -j }(z|x-y|)^{2+}\chi_{\{z|x-y|>1\}}\big]\\ \lesssim z^{\ell-j} (|x-y|^{\ell}+ \log^-|x-y|). \qedhere
	\end{multline*}

\end{proof}

\section{Free Dirac dispersive estimates}\label{sec:free}

Due to the relationship between the free Dirac evolution and the
free Klein-Gordon equation, $D_m^2 f=(-\Delta+m^2)f$,
we can expect a natural time decay rate of size $|t|^{-\f12}$ 
as one has in the wave equation (when $m=0$) or   Klein-Gordon equation (when $m>0$), provided the initial data has more than $\f 32$ weak derivatives in $L^1(\R^2)$.
In the case of Dirac equation, as in Klein-Gordon, the time decay can be improved to $\la t\ra^{-1}$ for smoother initial data.   In particular, we have  the following theorem bounding the evolution  between classical Besov spaces. 
 
\begin{theorem}\label{thm:free high} Fix $j\in \N$. Let $\chi_j$ be a smooth cut-off for the set $\lambda\approx 2^j$. Then,   
the free Dirac equation satisfies: 
	\begin{align}\label{eqn:freeDirac high}
		\|e^{-itD_m} \chi_j(D_m) f\|_{L^\infty} \les 
		 \min\left(2^{2j}, 2^{\frac{3j}2}|t|^{-1/2}, 2^{2j} |t|^{-1}\right)
		\|f\|_{L^1}.
	\end{align}
Let  $ \chi_0(\lambda) $ be a smooth cut-off for a small neighborhood of $m$.   Then,   
we have
	\begin{align}\label{eqn:freeDirac low}
		\|e^{-itD_m} \chi_0(D_m) f\|_{L^\infty} \les 
		 \frac{1}{\langle t\ra }
		\|f\|_{L^1}.
	\end{align}
\end{theorem}

This estimate can best be viewed as a mapping on Besov spaces:\footnote{ For $s\in \R$ and $1\leq p,q\leq \infty$ , we define the Besov space $B_{p,q}^s(\R^2, \C^2)$ to be the space of all tempered distributions $f$ such that
	$$
		\|f\|_{B^s_{p,q}}=\left(\|P_{<1}f\|_p^q+\sum_{j=1}^\infty 2^{jsq} 
		\|P_j   f\|_p^q	\right)^{\f1q}<\infty,
	$$
	where $P_{<1}$ and $P_j$ are the usual Littlewood-Paley projections.}
 \begin{corollary}

	Under the assumptions of Theorem~\ref{thm:free high}, for any $\theta\in [0,\f12]$ and $s-s'\geq \f32+\theta$, we have the following bounds
	$$
		\|e^{-itD_m}\|_{B^s_{1,1}\to B^{s'}_{\infty,1}} \leq C
		\left\{\begin{array}{cc}
			|t|^{-\f12+\theta } & 0<|t|<1\\
			|t|^{-\f12-\theta } & |t|\geq 1
		\end{array}
		\right.
	$$

\end{corollary}

First note that (with $\lambda=\sqrt{z^2+m^2}$)
\begin{align*}
  \int_m^\infty e^{-it\lambda} \chi_j(\lambda) [\mathcal R_0^+(\lambda)
	-\mathcal R_0^-(\lambda)] \, d\lambda
	& =  \int_0^\infty e^{-it\sqrt{z^2+m^2}}
	\frac{z \chi_j(z) [\mathcal R_0^+(\lambda)
	-\mathcal R_0^-(\lambda)]}{\sqrt{z^2+m^2}} \, dz.
\end{align*}
Using \eqref{eq:dr1}, the formula  $[R_0^+-R_0^-](z^2)(x,y)=\frac{i}{2}J_0(z|x-y|)$, and
the asymptotics for the Bessel function, see \cite{AS},  we can write 
\begin{multline*}
\frac{\mathcal R_0^+(\lambda)
	-\mathcal R_0^-(\lambda)}{\sqrt{z^2+m^2}}=
	\frac{(-i\alpha \cdot \nabla +m\beta +\sqrt{z^2+m^2}I)}{\sqrt{z^2+m^2}} [R_0^+
	-R_0^-](z^2)(x,y) \\ =	\left\{\begin{array}{ll}
		\widetilde O_1(z/\sqrt{z^2+m^2}) & z|x-y|\ll 1\\
		e^{iz|x-y|}\widetilde \omega_+(z|x-y|)
		+e^{-iz|x-y|}\widetilde \omega_-(z|x-y|) & z|x-y|\gtrsim 1
	\end{array}
	\right.
\end{multline*}
where  $\widetilde \omega_\pm(z|x-y|)$ satisfies the same properties as $\omega_\pm(z|x-y|)$ in \eqref{R0large}.
Therefore, it suffices to consider the integrals 
\be\label{eq:free1stint}
	\int_0^\infty e^{-it\sqrt{z^2+m^2}}
 z \chi_j(z)  
	 \widetilde O_1(z/{\sqrt{z^2+m^2}}) 
	\, dz,
\ee
and 
\be\label{eq:free2ndint}
	\int_0^\infty e^{-it\sqrt{z^2+m^2}\pm iz|x-y|}
	 z \chi_j(z) 
	 \widetilde \omega_\pm(z|x-y|)\widetilde \chi(z|x-y|)
	\, dz,
\ee
The integral \eqref{eq:free1stint} is   $O( \min(2^{2j},2^j/|t|))$. The first bound follows since the integrand is bounded by $z\chi_j(z)$, and the second bound follows from an  integration by parts.
To estimate the integral in \eqref{eq:free2ndint}, we apply stationary phase method using   the following (slightly modified) lemma from
\cite{Sc2},
\begin{lemma}\label{stat phase}

	Let $\phi'(z_0)=0$ and $1\leq \phi'' \leq C$.  Then,
$$
    		\bigg| \int_{-\infty}^{\infty} e^{-it\phi(z)} a(z)\, dz \bigg|
    		\lesssim \int_{|z-z_0|<|t|^{-\frac{1}{2}}} |a(z)|\, dz 
    		+|t|^{-1} \int_{|z-z_0|>|t|^{-\frac{1}{2}}} \bigg( \frac{|a(z)|}{|z-z_0|^2}+
    		\frac{|a'(z)|}{|z-z_0|}\bigg)\, dz.
$$

\end{lemma}

Using Lemma~\ref{stat phase} with $\phi_\pm(z)=\sqrt{z^2+m^2}\mp \frac{z r}{t}$,  $r:=|x-y|$, we will prove the  following lemma which yields the desired bound for the integral in \eqref{eq:free2ndint} for $j=0$. Let $\widetilde\chi$ be a smooth cut-off for $[1,\infty)$ supported in $(1/2,\infty)$.

\begin{lemma}\label{lem:high stat phase}   If 
	$$
		|a(z)|\les 	\frac{z\chi(z) \widetilde \chi(zr)}{(1+zr)^{\f12}} ,
		\qquad |\partial_z a(z)|\les 
		\frac{\chi(z) \widetilde \chi(zr)}{(1+zr)^{\f12}},
	$$
	then we have the bound
	\begin{align*}
		\bigg|\int_0^\infty e^{-it \phi_\pm(z) } a(z)\, dz\bigg| \les \la t\ra^{-1},
	\end{align*}
where $\phi_\pm(z)=\sqrt{z^2+m^2}\mp \frac{z r}{t}$.
\end{lemma}

\begin{proof}

	The integral in question is clearly bounded.
	We need only show that for large $t$, the integral 
	can be bounded by $|t|^{-1}$.  We assume $t>0$ and treat only the case of
	$\phi_+$, in which case the critical point occurs 
	in $[0,\infty)$.  The case of $t<0$ can be treated
	with the argument below by interchanging the phases
	$\phi_\pm$. Note  that the critical 
point of $\phi_{+}$ occurs at $z_0=  mr/\sqrt{t^2-r^2}$.
From this, we can assume that $t>2r$, say,  for $z_0$ to be in a small neighborhood of the support of $\chi(z)$. Thus, we have $z_0\approx r/t$.

	We employ Lemma~\ref{stat phase}. First, consider
	the integral
	\begin{align}\label{eqn:astat1}
		\int_{|z-z_0|<t^{-\f12}} |a(z)|\, dz\les \int_{|z-z_0|<t^{-\f12}} \frac{z^\f12 \widetilde \chi(zr)}{ r^{\f12}}\, dz.
	\end{align}
	We consider cases based on the size of $z_0$ compared
	to $t^{-\f12}$.  First, consider the case when $z_0\gtrsim t^{-\f12}$.  Then, we have $z\les z_0 \approx r/t$, and hence
	\begin{align*}
		\eqref{eqn:astat1}\les t^{-\f12} \frac{z_0^{\f12}}{r^{\f12}}\les  t^{-1}.
	\end{align*}
	In the second case, we have $z_0\les t^{-\f12}$, which
	yields $z\les t^{-\f12}$.  In
	this case, we see
	\begin{align*}
		\eqref{eqn:astat1} \les \int_0^{t^{-\f12}}
		\frac{z^{\f12} \widetilde \chi(zr)}{r^{\f12}}\, dz
		\les \frac{t^{-\f34}}{r^{\f12}}.
	\end{align*}
	We also  note that the integral is zero unless $r\gtrsim t^{1/2}$, which provides the desired bound of $t^{-1}$.
	
	We now proceed to bound the contribution of
	\begin{align}\label{eqn:astat2}
		\int_{|z-z_0|>t^{-\f12}} \frac{ |a(z)|}{|z-z_0|^2}\, dz\les \int_{|z-z_0|>t^{-\f12}} \frac{ z^\f12 \widetilde \chi(zr)}{r^\f12 |z-z_0|^2} \, dz.
	\end{align}
	We only need show that this integral is bounded.
	
	We first consider the case when $z_0\ll t^{-\f12}$, in
	which case we have $|z-z_0|\approx z$, and hence
	\begin{align*}
		\eqref{eqn:astat2} \les    \int_{\R} 
		\frac{\widetilde \chi(zr)}{z^{\f32}r^{\f12}}\, dz
		\les 1.
	\end{align*}
	
	In the second case, when $ t^{-\f12}\les z_0$, since $z_0 \approx r/t$, we have $r\gtrsim t^{\f12}$.   With the change of variable $s=z+z_0$, we have
	\begin{align*}
		\eqref{eqn:astat2}  \les   \frac{1}{r^{\f12}} \int_{|s|>t^{-\f12}}
		\frac{s^{\f12}+z_0^{\f12}}{s^2} \, ds 
		 \les r^{-\f12}t^{\f14}+r^{-\f12}t^{\f12}z_0^{\f12}
		\les 1.
	\end{align*}

	Finally, we turn to the contribution of
	\begin{align}\label{eqn:astat3}
		\int_{|z-z_0|>t^{-\f12}}\frac{|a'(z)|}{|z-z_0|}\, dz \les \int_{|z-z_0|>t^{-\f12}}\frac{ \widetilde \chi(zr)}{z^{\f12}r^{\f12}|z-z_0|} \, dz.
	\end{align} 
	If $z_0\ll t^{-\f12}$, we see that $|z-z_0|\approx z$
	and similar to the treatment for \eqref{eqn:astat2}
	in this case, we have
	\begin{align*}
		\eqref{eqn:astat3} \les \int_{\R} 
		\frac{\widetilde \chi(zr)}{z^{\f32}r^{\f12}}\, dz
		\les 1.
	\end{align*}
	
	If $z_0\gtrsim t^{-\f12}$, we have $r\gtrsim t^{1/2}$ as above. 
	We calculate
	\begin{align*}
		\eqref{eqn:astat3}   \les \frac{1}{r^{\f12}}\bigg[
		\int_{|z-z_0|>t^{-\f12}}\frac{dz}{|z-z_0|^{\f32}}
		+\int_{\R} \frac{\widetilde \chi(zr)}{z^{\f32}}\, dz
		\bigg]\les r^{-\f12}t^{\f14}+1 \les 1.
	\end{align*}

	For completeness, we note that in the case of the phase
	$\phi_-$, the critical point occurs outside of $[0,\infty)$, and we have $\phi_-'(z)=\frac{z}{\sqrt{z^2+m^2}}+\frac{r}{t}\gtrsim z$, we have
	\begin{align*}
		\bigg|\int_0^\infty e^{-it\phi_-(z)}a(z)\, dz		\bigg|&\les t^{-1}\bigg(\int_0^\infty \frac{|a(z)|}{|\phi_-'(z)|^2}+\frac{|a'(z)|}{|\phi_-'(z)|}\, dz	\bigg).
	\end{align*}
	Using the bounds for $a(z)$ and $a'(z)$, we may bound
	this by
	\begin{align*}
		t^{-1}\int_0^\infty \frac{\widetilde \chi(zr)}{z^{\f32}r^{\f12}} \, dz \les t^{-1}
	\end{align*}
	as desired. This also takes care of the case when the critical point occurs outside a neighborhood of the support of $\chi$.
\end{proof}

 The following lemma finishes the proof of Theorem~\ref{thm:free high} by establishing the required bound for the integral in \eqref{eq:free2ndint} when $j\geq 1$.

\begin{lemma}\label{lem:high stat phase2} Fix $j\in\N$, and let $\chi_j(z)$ be  a cut-off to
$z\approx 2^j$.  If 
	$$
		|a(z)|\les 	\frac{z\chi_j(z) \widetilde \chi(zr)}{(1+zr)^{\f12}} ,
		\qquad |\partial_z a(z)|\les 
		\frac{\chi_j(z) \widetilde \chi(zr)}{(1+zr)^{\f12}},
	$$
	then we have the bound
	\begin{align*}
		\bigg|\int_0^\infty e^{-it \phi_\pm(z) } a(z)\, dz\bigg| \les  \min(2^{2j}, 2^{\frac{3j}2}|t|^{-1/2}, 2^{2j} |t|^{-1}),
	\end{align*}
where $\phi_\pm(z)=\sqrt{z^2+m^2}\mp \frac{z r}{t}$.
\end{lemma}
\begin{proof} First note that the integral is bounded by $2^{2j}$ since the integrand is bounded by $z\chi_j(z)$.
We now restrict ourselves to the case $t>0$ and only consider $\phi_+$.  We also take $m=1$ without loss of generality.
Let $\rho=2^{-j} z$, $q=2^jr$. We rewrite the integral as
\be\label{eq:temp1}
	2^{2j}\int_0^\infty e^{-it 2^{-j} \widetilde \phi(\rho) } \widetilde a(\rho)\, d\rho,
\ee
where $\widetilde \phi(\rho)= 2^j (\sqrt{2^{2j}\rho^2+1}- \frac{\rho q}{t})$. Moreover,  $\widetilde a$ satisfies the bounds that $a$ satisfies with $j=1$ and $r=q$. 
Also note that $\frac{d^2}{d\rho^2}{\widetilde{\phi}} \approx 1$, since $\chi_1$ is supported away from $0$.

The critical point of $\widetilde\phi$ is $\rho_0=\frac{2^{-j}q}{\sqrt{t^2 2^{2j}-q^2}}$. Therefore, $\rho_0$ is in the support of $\chi_1(\rho)$ provided that 
$$
|t-q2^{-j}|\approx q2^{-3j},
$$   
which implies that $q\approx 2^j t$. This implies that
$$
		|\widetilde a(\rho)|\les 	\frac{\chi_1(\rho) \widetilde \chi(\rho q)}{ \sqrt{2^{j} t} },
		\qquad |\partial_\rho \widetilde a(\rho)|\les 
		\frac{\chi_1(\rho) \widetilde \chi(\rho q)}{ \sqrt{2^{j} t}}.
$$
Using the first inequality above in \eqref{eq:temp1} directly, we can bound \eqref{eq:temp1} by  $  2^{3j/2}t^{-1/2}$.
On the other hand, using these bounds in Lemma~\ref{stat phase} with $t2^{-j}$ instead of $t$. We bound \eqref{eq:temp1} by
$$
    		 2^{2j} \int_{|\rho-\rho_0|<\sqrt{2^j/t} } \frac{1}{ \sqrt{2^{j} t}}\, d\rho
    		+\frac{2^{3j}}{t}\int_{|\rho-\rho_0|>\sqrt{2^j/t}} \frac{1 }{ \sqrt{2^{j} t} }\bigg( \frac{1}{|\rho-\rho_0|^2}+
    		\frac{\chi_1(\rho) }{|\rho-\rho_0|}\bigg)\, d\rho \les \frac{2^{2j}}{t}.
$$

In the cases $q\ll 2^j t$ or $q\gg 2^j t$, we have $|\f{d}{d\rho}\widetilde\phi| \gtrsim 2^{2j}$. Therefore, an integration by parts implies that the integral is bounded by $2^j/t$. 
\end{proof}

\section{Perturbed resolvent expansions around the threshold energy $m$}\label{sec:m exp}
 
In this section, using Lemma~\ref{lem:R0exp}, we  develop expansions for the perturbed resolvents $\mR_V^\pm(\lambda)$ near the threshold $\lambda=m$ in  the case when the threshold is regular, and when there is an s-wave resonance at the threshold, see   Definition~\ref{resondef}. 
  
Since  the matrix 
$V:\mathbb R^2 \to \mathbb C^{2}$ is self-adjoint,
the spectral theorem allows us to write
$$
	V=B^*\left(\begin{array}{cc}
		\lambda_1 & 0 \\ 0 &\lambda_2
	\end{array}\right)B
$$
with $\lambda_j \in \mathbb R$.  We 
further write $\eta_j =|\lambda_j|^{\f12}$,
\begin{align*}
	V=B^*\left(\begin{array}{cc}
	\eta_1 & 0 \\ 0 & \eta_2
	\end{array}\right) U \left(\begin{array}{cc}
	\eta_1 & 0 \\ 0 & \eta_2
	\end{array}\right)B = v^*Uv,
\end{align*}	
	where
\begin{align}\label{eq:vabcd} 
	U=\left(\begin{array}{cc}
	\textrm{sign}(\lambda_1) & 0 \\ 0 & \textrm{sign}(\lambda_2)
	\end{array}
	\right),\,\,\,\text{ and }\,\,
 v=\left(\begin{array}{cc}
		a & b\\ c &d
	\end{array}\right):=\left(\begin{array}{cc}
	\eta_1 & 0 \\ 0 & \eta_2
	\end{array}\right)  B.
\end{align}
Note that the entries of $v$ are $\les \la x\ra^{-\beta/2}$, provided that the entries of $V$ are $\les \la x\ra^{-\beta}$. This representation of $V$ allows us to employ the
symmetric resolvent identity to write the perturbed
resolvent $\mR_V(\lambda)=(D_m+V-\lambda)^{-1}$ as (with $\lambda=\sqrt{m^2+z^2}$, $0<z\ll 1$)
\begin{align}\label{symmresid}
	\mR_V(\lambda)=\mR_0(\lambda)-\mR_0(\lambda)v^* (U+v\mR_0(\lambda)v^*)^{-1}v\mR_0(\lambda).
\end{align}
Our goal is to invert the operator
\be \label{eq:Mpm}
M^\pm (z)=U+v\mR_0^\pm \left(\sqrt{m^2+z^2}\right)v^*, \,\,\,\,0<z\ll 1.
\ee 
Recall that $\mR_0(\lambda)(x,y) =2m  g^\pm (z) M_{11}+\mG_0(x,y)
		+E_0^{\pm}(z)(x,y)$.
Therefore,
$$
M^\pm (z)=U+v\mG_0v^* + 2m  g^\pm (z) vM_{11}v^*+vE_0^{\pm}(z) v^*.
$$

Recalling \eqref{eq:vabcd}, for $f=(f_1,f_2)^T\in L^2 \times L^2$, we have 
 \begin{align*}
	vM_{11}v^*f(x)
	 =\left(\begin{array}{c}
		a(x)  \\ c(x) 
		\end{array}\right)  
		\int_{\R^2}\overline  a(y) f_1(y)+\overline{c}(y)f_2(y)\, dy.
\end{align*}
Thus, we arrive at
\begin{align}\label{vM11v}
	vM_{11}v^*=\|(a,c)\|_2^2 P,
\end{align}
where $P$ is the projection onto the vector $(a,c)^T$.  We also define the
operators $Q:=1-P$, $T:=U+v\mathcal G_0 v^*$, and let
$$
		\mathbbm g^{\pm}(z):=2m \|(a,c)\|_2^2 g^\pm(z).
$$
 We have 
\begin{lemma} \label{lem:M_exp}
	For $ 0<z\ll 1$, we have 
	\begin{align*}
		M^{\pm}(z)= \mathbbm g^\pm (z) P+ T+M_0^{\pm}(z),
	\end{align*}
where,\footnote{The Hilbert-Schmidt norm of an integral operator
$K$ with integral kernel $K(x,y)$ is defined by
$$
	\|K\|_{HS}^2= \int_{\R^4} |K(x,y)|^2\, dx\, dy.
$$} for any $\frac{1}{2}\leq k<2$,
	\begin{align*}
		\big\| \sup_{0<z\ll 1} z^{j-k} | \partial_z^j M_0^{\pm}(z)(x,y)|\big\|_{HS}\les 1,\,\,\,\,j=0,1,
	\end{align*}
	if $|v_{ij}(x)|\lesssim \langle x\rangle^{-\beta}$ for some $\beta>1+k$.
	
Moreover,
	\begin{align}\label{M0_defn}
		M_0^{\pm}(z)=  g_1^\pm (z) v\mathcal G_1v^* +  z^2 v\mathcal G_2v^*+M_1^{\pm}(z),
	\end{align}
	where, for any $2<\ell<4$,
	\begin{align*}
		\big\| \sup_{0<z\ll 1} z^{j-\ell} | \partial_z^j M_1^{\pm}(z)(x,y)|\big\|_{HS}\les 1,\,\,\,\,\,j=0,1,
	\end{align*}
	if $\beta>1+\ell$.
\end{lemma}

\begin{proof} Note that by \eqref{eq:Mpm}, Lemma~\ref{lem:R0exp}, and the discussion above, we have
\be\label{M0M1}
M_0=vE_0v^*,\,\,\,M_1=vE_1v^*.
\ee
Therefore the statement for $j=0,1$ follows from the error bounds in Lemma~\ref{lem:R0exp}, and the fact that 
$ (|x-y|^{\ell }+ \log^-|x-y|) \la x\ra^{-\beta} \la y\ra^{-\beta}$ is a Hilbert-Schmidt kernel for $\beta>1+\ell$ and $\ell>-1$. 
\end{proof}

We employ the following terminology, following \cite{Sc2,eg2,eg3}

\begin{defin}
	We say an operator $T:L^2\times L^2(\R^2)\to L^2\times L^2(\R^2)$ with kernel
	$T(\cdot,\cdot)$ is absolutely bounded if the operator with kernel
	$|T(\cdot,\cdot)|$ is bounded from $L^2\times L^2(\R^2)$ to $L^2\times L^2(\R^2)$.
\end{defin}

We note that Hilbert-Schmidt and finite-rank operators are absolutely bounded operators.

As in the case of the Schr\"odinger operator, the invertibility of the leading term of $M$ depends on the regularity of the threshold energy.   Here we  give the definition of threshold  resonances.  Later, in Section~\ref{sec:resonanceclass}, we study the classification of these resonances in detail.

\begin{defin}\label{resondef}\begin{enumerate}
\item Let $Q=1-P$. We  say that $\lambda=m$ is a regular point of the spectrum of
	$H=D_m+V$ provided that $QTQ=Q(U+v\mG_0v^*)Q$ is invertible on $Q(L^2\times L^2)$.
	If $QTQ$ is invertible, we denote $D_0:=(QTQ)^{-1}$ as an operator
	on $Q(L^2\times L^2)$.

\item Assume that $m$ is not a regular point of the spectrum. Let $S_1$ be the Riesz projection
onto the kernel of $QTQ$ as an operator on  $Q(L^2\times L^2)$.
Then $QTQ+S_1$ is invertible on  $Q(L^2\times L^2)$.  Accordingly, with a slight abuse of notation we redefine $D_0=(QTQ+S_1)^{-1}$ as an operator
on  $Q(L^2\times L^2)$.
We say there is a resonance of the first kind at $m$ if the operator $T_1:= S_1TPTS_1$ is invertible on
$S_1(L^2\times L^2)$.

\item We say there is a resonance of the second kind at  $m$  if $T_1$ is not invertible on
$S_1(L^2\times L^2)$ but
$T_2:=S_2v\mG_1v^*S_2$ is invertible
on $S_2(L^2\times L^2)$, where $S_2$ is the Riesz projection onto the kernel of $T_1$. Recall the definition of
$\mG_1$ and $\mG_2$ in \eqref{G1 def} and \eqref{G2 def}.

\item Finally, if $T_2$ is not invertible on $S_2(L^2\times L^2)$, we say there is a resonance of the third kind at  $m$.
We note that in this case the operator $T_3:=S_3v\mG_2v^*S_3$ is always invertible on $S_3(L^2\times L^2)$, where $S_3$ is the Riesz
projection onto the kernel of $T_2$ (see Lemma~\ref{G2 invert} below). 

\end{enumerate}

\end{defin}

\begin{rmk} \label{rmk:resonances} i) Since $S_1\leq Q$, for any $\phi \in S_1$, $P\phi=0$, i.e.
$$M_{11}v^*\phi=0.$$
ii) Note that $v\mathcal{G}_0v^*$ is compact and self-adjoint. Hence, $QTT$ is a compact perturbation of $QUQ$ and it is self-adjoint. Also, the spectrum of $QUQ$ is in $\{-1,1\}$. Hence, zero is the isolated point of the spectrum of $QTQ$ and $dim(Ker_{QTQ})$ is finite. Thus $S_1$ is a finite rank projection.\\
iii) As in the case of Schr\"odinger operator in $\R^2$ (see e.g. \cite{JN}),   the projections
$S_1-S_2$, $S_2-S_3$ and $S_3$  correspond  to   s-wave  resonances,  p-wave  resonances,
and  eigenspace at $m$ respectively. In particular, resonance of the first kind means that there is only an s-wave resonance at $m$.
Resonance of the second kind means that there is a p-wave resonance, and there may or may not be an s-wave resonance.
Finally, resonance of the third kind means that $m$ is an eigenvalue, and there may or may not be s-wave and p-wave resonances.
We characterize these projections in Section~\ref{sec:resonanceclass}. We will also prove, see Remark~\ref{rmk:ranks}, that the rank of $S_1-S_2$ is at most $1$ and the rank of $S_2-S_3$ is at most $2$.  \\
iv) Since $QTQ$ is self-adjoint, $S_1$ is the orthogonal projection onto the kernel of $QTQ$, and we have
(with $D_0=(QTQ+S_1)^{-1}$) $$S_1D_0=D_0S_1=S_1.$$
This statement also valid for $S_2$ and $(T_1+S_2)^{-1}$, and for $S_3$ and $(T_2+S_3)^{-1}$.\\
v) The operator $QD_0Q$ is absolutely bounded in $L^2\times L^2$, see Lemma~\ref{lem:QDQ abs} below. 
\\ 
vi) The operators with kernel $v\mG_k v^*$ are Hilbert-Schmidt operators on $L^2\times L^2$ if $|v_{ij}(x)|\les \langle x\rangle^{-\beta}$
for $\beta>\frac{3}{2}$ if $k=1$ and $\beta>3$ for $k=2,3$. However, $v\mG_0v^*$ is not Hilbert-Schmidt because of the local singularity of size $|x-y|^{-1}$. 
\end{rmk}

We can now use the expansions for $M^{-1}$ from the papers \cite{Sc2,eg2,eg3}  since $M$ has the same form with the same error bounds, and with analogous definitions for $S_j$. We include these expansions without proof.  
 
\begin{lemma}\label{lem:Minverseregular} Assume that $m$ is a regular point of the spectrum of $H$. Also assume that $|v_{ij}(x)|\les \la x\ra^{-\frac32-}$. Then 
$$(M^\pm(z))^{-1}=h^\pm(z)^{-1} S +QD_0Q+E^\pm(z)$$
where 
$$S=\left[\begin{array}{cc}
P& -PTQD_0Q\\ -QD_0QTP &QD_0QTPTQD_0Q
\end{array} \right],$$ 
 $h^\pm(z)=\mathbbm g^\pm(z)+\,$trace$\,(PTP-PTQD_0QTP)$,
 $S$ is a self-adjoint, finite rank operator, and
\begin{align*}
		\big\| \sup_{0<z\ll 1} z^{j-1/2} | \partial_z^j E^{\pm}(z)(x,y)|\big\|_{HS}\les 1,\,\,\,\,j=0,1,
	\end{align*}
\end{lemma} 
\begin{lemma}\label{lem:Minversefirstkind} Assume that there is a resonance at $m$ of the first kind. Also assume that $|v_{ij}(x)|\les \langle x\rangle^{-\frac32-}$. Then 
\begin{multline*}
		M^{\pm}(z)^{-1} =-h_{\pm}(z)S_1D_1S_1-SS_1D_1S_1-S_1D_1S_1S\\
		 -h_{\pm}(z)^{-1}SS_1D_1S_1S+
		h_{\pm}(z)^{-1} S+QD_0Q +E^\pm(z),
\end{multline*}
 Here  $E^\pm(z)$, $S$, and  $h^\pm(z)$  are as in the previous lemma  with $D_0=Q(T+S_1)^{-1}Q$, and $D_1=T_1^{-1}=(S_1TPTS_1)^{-1}$.  
\end{lemma} 
\begin{rmk}\label{rmk:LAP}
One can also obtain analogous espansions\footnote{These expansions would require more decay from the potential then we have in Assumption~\ref{assumption}.}  in the cases when there  is a resonance of the second or third kind as in the Schr\"odinger equation, \cite{eg2}. We chose not to state these expansions explicitly since we are not considering dispersive estimates in these cases. By substituting the expansions for  $M^{\pm}(z)^{-1}$ in Lemmas~\ref{lem:Minverseregular} and \ref{lem:Minversefirstkind} into \eqref{symmresid}, we obtain expansions for the resolvent  showing that $ (\lambda-m) \mR_V(\lambda)$ is uniformly bounded between weighted $L^2$ spaces in a neighborhood of $m$.  This implies that  there are no eigenvalues in a neighborhood of $m$. In particular, there are only finitely many eigenvalues in the spectral gap $(-m,m)$.  It also implies a limiting absorption principle bound around the threshold.
\end{rmk}

\section{Low energy dispersive estimates} \label{sec:low}

In this section we study the low-energy part of the
perturbed Dirac evolution.    For technical reasons, which we detail below, we consider the evolution as an operator from $H^1 $ to $BMO $.

\begin{theorem}\label{thm:low energy}

	Under  Assumption~\ref{assumption} part i), with $\chi$ a smooth cut-off to a sufficiently small neighborhood of the threshold energy
	$\lambda=m$.
	We have the dispersive bound 
	\begin{align*}
		\| e^{-itH}P_{ac}(H)\chi(H) f\|_{BMO }\les
		\frac{1}{\la t\ra} \|f\|_{H^1 }.
	\end{align*}
	This bound holds if $\lambda=m$ is regular or if there is
	a resonance of the first kind.

\end{theorem}

As usual, we prove this bound by considering the Stone's formula, \eqref{Stone}.  
In the case there is a resonance at $m$ of the first kind, using Lemma~\ref{lem:Minversefirstkind}  in   \eqref{symmresid}, we have 
\begin{multline}\label{eq:resolvem}
		\mR_V(\lambda)=\mR_0(\lambda)\\-\mR_0(\lambda)v^*  \big[ -h_{\pm}(z)S_1D_1S_1 +A  
		 +h_{\pm}(z)^{-1} (S-SS_1D_1S_1S)   +E^\pm(z) \big]v\mR_0(\lambda),
\end{multline}
where $A:=QD_0Q-SS_1D_1S_1-S_1D_1S_1S$.
Since this expansion contains the terms arising in the  regular case, it suffices to prove the dispersive estimate in the case of a resonance of the first kind.  We bound the contribution of each operator in this expansion in a series of technical propositions.  The first term containing only a single free resolvent $\mR_0$ is controlled by the bound in Theorem~\ref{thm:free high}, specifically \eqref{eqn:freeDirac low}.

To control the contributions to the Stone's formula, using \eqref{eq:dr1}, \eqref{eq:freschroresolv}, and \eqref{R0large}, we write the outermost resolvents when $0<z \ll 1$ as
\begin{align}
\mR_0^\pm(\lambda)(x,y) 
 = &[-i\alpha \cdot \nabla + e(z)]R_0^\pm(z^2)(x,y) + 2mI_1 R_0^\pm(z^2)(x,y) \nn \\
 = &-\frac1{2\pi}\chi(z|x-y|) [-i\alpha \cdot \nabla ] \log(z|x-y|)\nn \\
 &+\chi(z|x-y|)[-i\alpha \cdot \nabla  ] \left(R_0^\pm(z^2)(x,y)+\frac1{2\pi}\log(z|x-y|)\right)\nn\\
 &+  \chi(z|x-y|)  e(z)  R_0^\pm(z^2)(x,y) \nn\\ & + \widetilde\chi(z|x-y|) e^{\pm i z|x-y|}  \omega_1^\pm(z(x-y))+ 2mI_1 R_0^\pm(z^2)(x,y)\nn\\
 =:& R_1 +R_2^\pm+R_3^\pm+R_4^\pm+R_5^\pm. \label{eqn:R1R2R3}
\end{align}
Here $e(z)=\widetilde O_1(z^2)$, and it does not have  $\pm$ dependence. Further, $\omega_1^\pm(z(x-y))$ satisfies the same bounds as $z \omega^\pm(z|x-y|)$.  We note that these expansions differ slightly from those in Sections~\ref{sec:exps} and \ref{sec:m exp}, as we tailor them to prove the dispersive bounds rather than to develop expansions for $M^\pm(z)^{-1}$.

We note that the dispersive bounds for the term  containing only $R_5^\pm$ is identical to the ones given for the Schr\"odinger operator in \cite{Sc2} and \cite{eg2}, since $R_5^\pm$ satisfies the same bounds and cancellation properties as the Schr\"odinger resolvent $R^\pm_0$. Moreover,  the corresponding orthogonality property 
\begin{align}\label{eqn:Qtrick}
QvM_{11}=M_{11}v^* Q=0
\end{align}
holds because of the projection $I_1$. The slight difference in the phase in Stone's formula can be taken care of using Lemma~\ref{lem:high stat phase} in place of Lemma~2 in \cite{Sc2}. The contribution of the terms containing $R_2^\pm$ and $R_3^\pm$ in addition to  $R_5^\pm$ are easier since $R_2^\pm$ and $R_3^\pm$ satisfy the same bounds as $F$ or $G$ from Lemma 3.3 in \cite{eg2} (also see \cite{Sc2}). Therefore one does not need the orthogonality property for these terms. 

Thus, it suffices to consider the terms containing $R_1$ or $R_4^\pm$ on the left. We will write the operator $\mR_0^\pm$ on the right as $\mR^\pm_L+\mR^\pm_H$, where
\be\label{eq:RLRH}\left\{\begin{array}{l}
\mR_L^\pm(\lambda)(x,y)=\chi(z|x-y|) \mR_0^\pm(\lambda)(x,y), \\
\mR_H^\pm(\lambda)(x,y)=\widetilde\chi(z|x-y|) \mR_0^\pm(\lambda)(x,y) = e^{\pm i z |x-y|} \widetilde\omega_\pm(z(x-y)).\end{array} \right.
\ee
Before we bound the contribution of these terms to the
Stone formula, \eqref{Stone}, we note that the operator $R_1$ is not bounded as an operator from $L^1\to L^2$ or from  $L^2\to L^\infty$.  This is an important technical difference from the analysis of Schr\"odinger operators in \cite{Sc2,eg2}.  One can iterate the standard resolvent identity $\mR_V=\mR_0-\mR_0 V\mR_V$ to smooth out the local singularity and obtain a bound from $L^1\to L^\infty$, 
though this would cause the time decay to be of the form $|t|^{-1}(\log t)^k$ for some $k>0$ for large $t$ due to the leading $\log \lambda$ behavior of the free resolvent, see Lemma~\ref{lem:R0exp}.  Instead, we consider the Dirac evolution as a mapping from the Hardy space $H^1$ to $BMO$.    The following lemma is useful.

\begin{lemma}\label{lem:R1h1j}  For any $H^1 \times H^1(\R^2)$ atom $g$, and for $0<z\les 1$, we have 
\be\label{eq:R1h11} 
 v  R_1 g=  g_1+g_2\widetilde O_1(z),  
\ee
where $$\|g_j\|_{L^2}\les \|v\|_{L^2}+\|v\|_{L^\infty},\,\,j=1,2.$$
Furthermore,  
\be\label{eq:R1h13}
\sup_{0<z\ll 1} \left(\|v \mR_L^\pm  g \|_{L^2} + z \|v \partial_z \mR_L^\pm  g \|_{L^2}\right) \les       \|v\|_{L^2},
\ee
and
\be\label{eq:R1h14}
 \| \chi(z) (\mR_L^+-\mR_L^-)  g  \|_{L^\infty_x L^\infty_z}   \les    1, \ee
\be\label{eq:R1h15}
\left\| \chi(z) \partial_z (\mR_L^+ -\mR_L^-) g \right\|_{L^\infty_x L^1_z}\les  1.
\ee

\end{lemma}

\begin{proof}
We rewrite $R_1$ as
\begin{multline} \label{eq: R_1}
R_1=\frac{i}{2\pi} \chi(z|x-y|)\frac{\alpha\cdot(x-y)}{|x-y|^2}\\
=\frac{i}{2\pi} \frac{\alpha\cdot(x-y)}{|x-y|^2}-\frac{i}{2\pi} \widetilde\chi(z|x-y|)\frac{\alpha\cdot(x-y)}{|x-y|^2} 
= \frac{i}{2\pi} \frac{\alpha\cdot(x-y)}{|x-y|^2}+\widetilde O_1(z). 
\end{multline} 
The contribution of the first summand gives $g_1$.
By  Theorem~1 in \cite{dinglu}, the operator defined by the first term in $R_1$ is bounded from $H^1$ to $L^2$.
 Therefore, 
$$
\|g_1\|_{L^2}\leq \|v\|_{L^\infty} \Big\| \int_{\R^2}  \frac{\alpha\cdot (x-y)}{|x-y|^2} g(y) dy \Big\|_{L^2_x} \les \|v\|_{L^\infty}.
$$
The bound for $g_2$ is immediate from the expansion above. 

The second claim follows from 
the expansion
$$
	\mR_L(z)(x,y)=R_1-\frac{mI_1}{\pi}\log(z|x-y|)\chi(z|x-y|)
	+\widetilde O_1(1), \,\,\,\,\,\,\,\,\, |\partial_z \mR_L|\les z^{-1},
$$
the fact that $\chi(zx)\log(zx)\in BMO$ with norm independent of $z$, and $H^1$--$BMO$ duality.

To obtain the last two claims, note that  
\begin{multline}\label{eqn:RL difference}
[\mR_L^+-\mR_L^-](\lambda)(x,y)= \chi(z|x-y|)[-i\alpha \cdot \nabla + m \beta +\sqrt{m^2+z^2} I] J_0(z|x-y|)\\
=\chi(z|x-y|)(cI_1+ \widetilde O_1(z)),
\end{multline}
which immediately implies \eqref{eq:R1h14}. To obtain \eqref{eq:R1h15}, note that
$$
\int  \left(\int |x-y| |\chi^\prime(z|x-y|)| dz \right)  |I_1g(y)| dy \les \|g\|_{L^1}= 1. \qedhere
$$
\end{proof}


Recall that   $\log^-(y):=-\log (y) \chi_{\{0<y<1\}}$.  In addition, we define  
$\log^+(y)=\log(y)\chi_{\{y>1\}}$.

We start with the contribution of the terms $A:=QD_0Q  -SS_1D_1S_1-S_1D_1S_1S$ from \eqref{eq:resolvem}, for which we rely only on the absolute boundedness of the operator, and do not use any orthogonality properties of the projection operators $Q$ or $S_1\leq Q$.  
By symmetry and the discussion above, it suffices to consider the terms
\be \label{eq:gammaA}
\Gamma_1:=  R_1   v^*A v (\mR_L^+-\mR_L^-), \,\,\,\, \Gamma_2:=  R_1   v^*A v \mR_H^+, \,\,\,\,\Gamma_3:=  R_4^+ v^*
  A v  \mR_0^+,  
\ee

\begin{prop}\label{prop:A} Let $\Gamma_j$ be defined as in \eqref{eq:gammaA}. Then, under the assumptions of 
Theorem~\ref{thm:low energy},
for any $H^1 \times H^1(\R^2) $ atoms $f$, $g$, and for each $j=1,2,3 $ we have 
$$
\int_0^\infty e^{it\sqrt{z^2+m^2}} \frac{z\chi(z)}{\sqrt{z^2+m^2}} \la \Gamma_jf, g\ra dz = O(1/\la t \ra).
$$
\end{prop}
\begin{proof} We start with $\Gamma_1$. 
By an integration by parts we rewrite the integral above as
$$
-\frac{ie^{itm}}t\la \Gamma_1|_{z=0} f,g \ra +\frac{i}t\int_0^\infty e^{it\sqrt{z^2+m^2}}    \partial_z \big[\chi(z)  \la \Gamma_1f, g\ra \big] dz,
$$
where $\Gamma_1|_{z=0}$ means $\lim_{z\to 0+} \Gamma_1(z)$.
Therefore, we need to prove that 
$$
\big|\la \Gamma_1|_{z=0} f,g \ra \big|\les 1,
$$
and 
\be\label{eq:gammaibp}
 \big\|  \partial_z \big[\chi(z)  \la \Gamma_1f, g\ra \big] \big\|_{L^1_z}\les 1.
\ee
In fact, since $\chi(1)=0$, by the fundamental theorem of calculus, it suffices to prove \eqref{eq:gammaibp}.  
Using the bounds in Lemma~\ref{lem:R1h1j}, we have
$$
\left| \partial_z \big[\chi(z)  \la \Gamma_1f, g\ra \big]\right| \les (\|v\|_{L^\infty}+\|v\|_{L^2})\|v\|_{L^2}\||A|\|_{L^2\to L^2}. 
$$
The claim for small $t$ also follows from these bounds without integrating by parts. 

Now we consider $\Gamma_2$.   By the absolute boundedness of $A$ and Lemma~\ref{lem:R1h1j}, we have 
\be\label{eq:AvR1g}
AvR_1 g = \widetilde g_1+\widetilde g_2 \widetilde O_1(z),
\ee
for some $\widetilde g_1, \widetilde g_2\in L^2$. Using this we write the oscillatory integral as
$$
\int_{\R^4}f(y) v(x_1)  \int_0^\infty e^{it\sqrt{z^2+m^2}\pm i z|x_1-y|} a(z,y,x_1) dz dx_1 dy,
$$ 
where 
$$
a(z,y,x_1) =(\widetilde g_1(x_1)+\widetilde g_2(x_1)\widetilde O_1(z)) \frac{z\chi(z)}{\sqrt{z^2+m^2}}  \widetilde\chi(z|y-x_1|) \widetilde\omega_\pm(z(y-x_1)). 
$$
Therefore, using  Lemma~\ref{lem:high stat phase}, we bound the integral above by
\begin{align*}
\frac1{\la t\ra} \int_{\R^4} |f(y)| |v(x_1)| (|\widetilde g_1(x_1)|+|\widetilde g_2(x_1)|)  dx_1 dy   \leq \frac1{\la t\ra} \|f\|_{L^1} \|v\|_{L^2} (\|\widetilde g_1\|_{L^2}+\|\widetilde g_2\|_{L^2}).
\end{align*}
 
Finally we consider $\Gamma_3$. 
Using $\mR_0=\mR_L+\mR_H$, we bound the contribution of $\mR_H$ to the integral by (with $r=|x-x_1|$, $s=|y-y_1|$)
$$ 
\sup_{r,s}\left|\int_0^\infty e^{it\sqrt{z^2+m^2}\pm iz (r+s) } \frac{z\chi(z)}{\sqrt{z^2+m^2}} \widetilde\chi(zr) \omega_1^\pm(zr)  \widetilde\chi(zs) \widetilde\omega_\pm(zs) dz  \right|,
$$
which is $O(1/\la t\ra)$ by Lemma~\ref{lem:high stat phase} noting that
$$
\omega_1^\pm(zr)\widetilde\omega_\pm(zs) =\widetilde O_1\left((1+z(r+s))^{-1/2}\right). 
$$
For $\mR_L$ we note that 
$$
z \mR_L= zR_1+ \widetilde O_1(z^{1-}(1+\log^{-}(|x-y|)).
$$ 
Therefore, since $\omega_1^\pm$ behaves like $z \omega_\pm$,   the argument above for $\Gamma_2 $ takes care of the first summand. 
For the second summand, writing $k(y_1,y)=(1+\log^-(y_1-y))$ and
using  Lemma~\ref{lem:high stat phase} we bound its contribution by
\begin{multline*}
\frac1{\la t\ra} \int_{\R^8} |f(y)| |v^*(x_1)| |A(x_1,y_1)| |v(y_1)| k(y_1-y) |g(x)|   dx_1dy_1  dx dy\\  \les\frac1{\la t\ra} \sup_y \|v(\cdot) (1+\log^-(\cdot-y))\|_{L^2} \| |A| \|_{L^2\to L^2} \|v\|_{L^2}.\qedhere
\end{multline*}

\end{proof}
 
Now we consider the 
contribution of the term $h_{\pm}(z)^{-1} S$   from \eqref{eq:resolvem}  (the contribution of  $h_{\pm}(z)^{-1} SS_1D_1S_1S$ is handled similarly).
By symmetry and the discussion above, it suffices to consider the terms
\be  \label{eq:gammaS}
\Gamma_1:=  R_1   v^*S v \left(\frac{\mR_L^+}{h_+(z)}-\frac{\mR_L^-}{h_-(z)}\right), \,\,\,\, \Gamma_2:=  h_+^{-1} R_1   v^* Sv \mR_H^+, \,\,\,\,\Gamma_3:=  h_+^{-1} R_4^+ v^*
 S v  \mR_0^+.  
\ee 
\begin{prop}\label{prop:S} The assertion of Proposition~\ref{prop:A} is valid for each $\Gamma_j$ defined in \eqref{eq:gammaS}.  
\end{prop}
\begin{proof}
The proof for $\Gamma_2$ and  $\Gamma_3$ follows from Proposition~\ref{prop:A} by noting that $h_\pm^{-1}=\widetilde O_1(1)$.
For $\Gamma_1$, it suffices to obtain the inequality \eqref{eq:gammaibp}. 
Note that using Lemma~\ref{lem:R0exp}, the identity $h^\pm(z)=c_1+g^\pm(z)+c_2$, and
 $$
 h^+(z)^{-1}-h^-(z)^{-1}= \widetilde O_1(\log^{-2}z)
$$
we have 
\begin{multline*}
\left(\frac{\mR_L^+}{h_+(z)}-\frac{\mR_L^-}{h_-(z)}\right)=  \chi(z|x-y|)  (R_1+cM_{11}) \widetilde O_1(\log^{-2}z) \\ + \chi(z|x-y|)\left[   -\frac{mI_1}{\pi} \log(|x-y|) \widetilde O_1(\log^{-2}z)  +  \widetilde O_1\big(z^{3/2} ( |x-y|^{3/2}+\log^-|x-y|)\big)\right].
\end{multline*}
The contribution of the first term is $O(\la t\ra^{-1})$ as in the proof of Proposition~\ref{prop:A} using \eqref{eq:R1h11}  for $f$ and $g$.   Using \eqref{eq:AvR1g},
the contribution of the the second term to the left hand side of \eqref{eq:gammaibp} can be bounded by 
\begin{align*}
\int_{\R^4} |v(y_1)| & |f(y)|  
\int_0^\infty \bigg|\partial_z\big[ (\widetilde g_1(y_1)+\widetilde g_2(y_1)\widetilde O_1(z)) \chi(z)\\  & \chi(z|y-y_1|)  \log(|y-y_1|) \widetilde O_1(\log^{-2}z)\big] \bigg| dz dy_1 dy \\
 \les & \int_{\R^4} |v(y_1)|  |f(y)| (|\widetilde g_1(y_1)|+|\widetilde g_2(y_1)|) \\
& \int_0^\infty \chi(z) \log(|y-y_1|) \left(\frac{\chi(z|y-y_1|)}{z\log^3(z)}+ \frac{|y-y_1|\chi^\prime(z|y-y_1|)}{ \log^2(z)}\right) dz dy_1 dy\\
 \les  &\bigg\| v(y_1)  
 \int_0^\infty \chi(z) (1+\log^-(|y-y_1|)) \bigg(\frac{\chi(z|y-y_1|)}{z\log^2(z)} +  |y-y_1|\chi^\prime(z|y-y_1|) \bigg) dz\bigg\|_{L^\infty_y L^2_{y_1} } 
\\
  \les & \left\|   v(y_1)   (1+\log^-(|y-y_1|))\right\|_{L^\infty_y L^2_{y_1} }  \les 1.
\end{align*}
In the second to last inequality, we used that $|\log z|^{-1}\les 1$ and
$$
	|\chi(z)\chi(z|y-y_1|)\log |y-y_1|\,| \les 1+\log^- |y-y_1|+|\log z|.
$$
The contribution of the last term can be handled similarly. 
\end{proof}

Now we consider the 
contribution of the error term $E^\pm(z)$  from \eqref{eq:resolvem}.
By symmetry and the discussion above, and dropping $\pm$ indices,  it suffices to consider the terms 
\be \label{eq:gammaEr}
\Gamma_1:=   R_1  v^* E v \mR_L,  \,\,\,\, \Gamma_2:=  R_1  v^* E v \mR_H,  \,\,\,\,  \Gamma_3:=  R_4 v^* E   v \mR_0. 
\ee
\begin{prop}\label{prop:Erlowlow} The assertion of Proposition~\ref{prop:A} is valid for each $\Gamma_j$ defined in \eqref{eq:gammaEr}.   
\end{prop}
\begin{proof}
For $\Gamma_1$, as in the proof of Proposition~\ref{prop:A} it suffices to prove that
$$
\big\|  \partial_z \big[\chi(z)  \la \Gamma_1f, g\ra \big] \big\|_{L^1_z}\les 1.
$$
 Note that
\be\label{temp:errorgamma}
\partial_z ( R_1  v^* E v \mR_L) =(\partial_z  R_1)  v^* E v\mR_L + R_1  v^* (\partial_z E) v\mR_L + R_1  v^* E v (\partial_z  \mR_L).
\ee
We only consider the contribution of the last summand, the others are similar.
We have 
$$
\partial_z\mR_L= \partial_z R_1 +  r\chi^\prime(zr)\log(zr) +  O(1/z)=  O\left(\frac1{r^{0+}z^{1+}}\right).
$$ 
We write (with $\widetilde E(x_1,y_1)  =\sup_z z^{-1/2}|E(z,x_1,y_1)|$)  
\begin{multline*}
\int \chi(z) |\la R_1  v^* E v (\partial_z   \mR_L  )f,g\ra| dz \\ 
\les \int \frac {\chi(z)}{z^{\frac12+}} |(v R_1 g)(z,x_1)| | \widetilde E(x_1,y_1) |  \frac{|v(y_1)|  }{|y-y_1|^{0+}}  |f(y)| dx_1 d y_1 d y dz.
\end{multline*}
By Cauchy-Schwarz and  Lemma~\ref{lem:R1h1j}, we estimate the $x_1$ integral by  
$$
\|(v R_1 g)(z,x_1)\|_{L^2_{x_1}} \|\widetilde E(x_1,y_1)\|_{L^2_{x_1}}\les \|\widetilde E(x_1,y_1)\|_{L^2_{x_1}},
$$
uniformly in $z$. Therefore, we estimate the integral above by
$$
 \int    \|\widetilde E(x_1,y_1)\|_{L^2_{x_1}} \frac{ |v(y_1)| }{|y-y_1|^{0+} }   |f(y)|  d y_1 d y   \les  \|f\|_{L^1} \|\widetilde E\|_{HS} \sup_y\|v/|y-\cdot|^{0+} \|_{L^2}  \les 1,
$$
where we used Cauchy-Schwarz in the $y_1$ integral.

The bound for the contribution of the first summand on the right hand side of \eqref{temp:errorgamma}   is nearly identical.  For the second summand,
one must use Lemma~\ref{lem:R1h1j} twice and use that
$\sup_z | z^{\f12}\partial_z E(z,x_1,y_1)|$ is Hilbert-Schmidt.

We now consider $\Gamma_2$, whose  contribution   to the
Stone formula is given by
\begin{multline*}
	\int_{\R^4}\int_0^\infty e^{it\sqrt{z^2+m^2}\pm i |y-y_1|} \frac{z\chi(z)}{\sqrt{z^2+m^2}} [vR_1g](z,x_1)
	E(z,x_1,y_1)\\ v(y_1) \widetilde \omega_\pm (z(y_1-y))
	f(y) \,  dx_1 dy_1 dy dz.
\end{multline*} 
We apply Lemma~\ref{lem:high stat phase} to this integral
with
$$
	a(z)=\frac{z\chi(z)}{\sqrt{z^2+m^2}}  vR_1g 
	Ev\widetilde \omega_\pm (z(y_1-y))
$$
we note that by Lemma~\ref{lem:R1h1j} we have
$$
	vR_1 g(z,x_1) = g_1(x_1)+ g_2(x_1) \widetilde O_1(z).
$$
for some $g_1,g_2\in L^2$.  This along with the bounds
on $E$ from Lemma~\ref{lem:Minverseregular}  and the definition of $\widetilde \omega_\pm$ yields the bound
\begin{align*}
	|a(z)| & \les (|g_1(x_1)|+|g_2(x_1)|)
	\sup_z \big|z^{-\f12} E(z,x_1,y_1)
	\big| |v(y_1)|\frac{z\chi(z)\widetilde \chi(z|y-y_1|)}{(1+z|y-y_1|)^{\f12}},\\
	|\partial_z a(z)|& \les  |g_2(x_1)| 
	\big(\sup_z \big|z^{-\f12} E(z,x_1,y_1)\big|  +\sup_z \big|z^{\f12} \partial_z E(z,x_1,y_1)
	\big|\big)  |v(y_1)|\frac{\chi(z)\widetilde \chi(|y-y_1|)}{(1+z|y-y_1|)^{\f12}}.
\end{align*}
This implies the desired time decay bound using Lemma~\ref{lem:high stat phase}.  The spatial integrals can be controlled as in the case of $\Gamma_1$.

For $\Gamma_3$, writing $\mR_0=\mR_L+\mR_H$, the contribution of $\mR_L$ follows as in the bounds of
$\Gamma_2$.  
For $\mR_H$, let $r=|x-x_1|+|y-y_1|$ and  
$$
a(z)=\frac{z \chi(z) E(x_1,y_1,z) }{h(z)\sqrt{z^2+m^2}} 
\widetilde\chi(z|y-y_1|) \widetilde\chi(z|x-x_1|) \widetilde\omega^\pm(z|x-x_1|)\widetilde\omega^\pm(z|y-y_1|).
$$
Note that $a(z)$ satisfies the bounds 
\begin{align*}
|a(z)|&\les  \frac{z \chi(z) \widetilde{\chi }(zr)  }{(1+zr)^{1/2}} \sup_z  | E(x_1, y_1,z)|,\\
|\partial_z a(z)|&\les \frac{ \chi(z) \widetilde{\chi }(zr)  }{(1+zr)^{1/2}} \sup_z \left(|  E( x_1,y_1,z)|  + |z  \partial_z  E(x_1, y_1,z)|\right).
\end{align*}
Therefore, using Lemma~\ref{lem:R1h1j} and then the bounds for $E$ and $\partial_z E$ given in Lemma~\ref{lem:Minverseregular},  we obtain the bound
$$
\frac1{\la t\ra} \int_{\R^8}  \sup_z \left(|E(x_1, y_1,z)|  + |z  \partial_z \mathcal E(x_1, y_1,z)|\right) |v(x_1)||v(y_1)| |f(x)|  |g(y)|  dx dy dx_1   dy_1 \les \frac1{\la t\ra}. \qedhere
$$

\end{proof}

To control the `s-wave' term with
$h^\pm S_1D_1S_1$ on the right hand side of \eqref{eq:resolvem}.  In particular, we need to consider terms of the
form
\begin{multline} \label{eq:gammaswave}
	\Gamma_1=[h^+(z)-h^-(z)] R_1 v^* S_1D_1S_1 v \mR_0^\pm, \,\,\,\, \\ \Gamma_2=
	h^+(z) R_1v^* S_1D_1S_1 v[\mR_0^+-\mR_0^-](z),\,\,\,\,\Gamma_3=h^+(z)R_4^+ v^* S_1D_1S_1 v \mR_0^+.
\end{multline}
\begin{prop}\label{prop:swave} The assertion of Proposition~\ref{prop:A} is valid for each $\Gamma_j$ defined in \eqref{eq:gammaswave}.  
\end{prop}
\begin{proof}
For $\Gamma_1$, note that $h^+(z)-h^-(z) =c$. Recalling Proposition~\ref{prop:A} for $\Gamma_2$ defined in \eqref{eq:gammaA}, it suffices to consider
the contribution of 
$$
R_1 v^* S_1D_1S_1 v \mR_L^+.
$$
Using \eqref{eq:R0exp}, we write
$$
 \mR_L^+(z)(x,y)=R_1-\frac{mI_1}{\pi}\log(z|x-y|)\chi(z|x-y|)  +\widetilde O_1(z^{3/2}(|x-y|^{3/2}+\log^- |x-y|)),
$$
we note that the contribution of the third summand follows from the analysis of $\Gamma_1$ in Proposition~\ref{prop:S}.
The contribution of the first summand is easier using Lemma~\ref{lem:R1h1j} for both $f$ and $g$. For the contribution of the second summand we need to use
the orthogonality $S_1vM_{11}=0$, which holds since $S_1\leq Q$. Let 
$$
F(z,y ,y_1)= -\frac{mI_1}{\pi}\log(z|y_1-y|)\chi(z|y_1-y|)+\frac{mI_1}{\pi}\log(z\la y\ra)\chi(z\la y\ra ).
$$
By the orthogonality $S_1vM_{11}=0$, and using Lemma~\ref{lem:R1h1j}, the contribution of the second summand to the left hand side of \eqref{eq:gammaibp} is given by
$$
 \left\|  \partial_z \left[\chi(z)   \int_{\R^6} [S_1D_1S_1](x_1,y_1) v(y_1) F(z,y_1,y) f(y) [g_1(x_1)+g_2(x_1) \widetilde O_1(z)](x_1) dx_1dy_1 dy \right]\right\|_{L^1_z},
$$
where $\|g_j\|_{L^2}\les 1$.
We have the following  bounds for $z\les z_0$ (see \cite{Sc2}, \cite[Lemma 3.3]{eg2}) 
\be\label{eq:Fbound}
|F(z,y_1,y)|\les \int_0^{z_0} |\partial_z F(z,y_1,y)| dz+ |F(0+,y_1,y)|\les k_2(y_1,y),
\ee
where $k_2(y_1,y):=1+\log^+(|y_1|)+\log^-(|y-y_1|)$. Therefore, we can estimate the integral above by
\begin{multline*}
\int_{\R^6} |[S_1D_1S_1](x_1,y_1)| |v(y_1) k_2(y_1,y)| |f(y)| (|g_1(x_1)|+|g_2(x_1)|)   dx_1dy_1 dy \\
\leq \left[ \sup_y \|v(\cdot) k_2(\cdot,y)\|_{L^2} \right] \||S_1D_1S_1|\|_{L^2\to L^2}  \| |g_1(x_1)|+|g_2(x_1)|\|_{L^2} \|f\|_{L^1}\les 1. 
\end{multline*}
The contribution of $\Gamma_3$ can be handled as in Proposition~\ref{prop:A} since the additional $z$ factor in $R_4$ kills the logarithm coming from $h(z)$.

For the contribution of $\Gamma_2$, we write
\begin{multline*}
[\mR_0^+-\mR_0^-](z)(x,y)=  [-i\alpha \cdot \nabla + m \beta +\sqrt{m^2+z^2} I] J_0(z|x-y|) \\ = 2mI_1J_0 + (R_4^+-R_4^-) + \widetilde O_1(z) \chi(z|x-y|) J_0.
\end{multline*}
The contribution of the last two summand is similar to the cases above. The contribution of the first summand can be handled using the orthogonality property as above and as in \cite{eg2}; the functions $G, \widetilde G$ from \cite{eg2} which have an additional factor of $z$ replace the function $F$ above.  The rest of the analysis is identical to the one above for low energies and to the analysis of the terms containing $R_4$ for the high energies. 
 \end{proof}
 
We are now ready to prove Theorem~\ref{thm:low energy}.  

\begin{proof}[Proof of Theorem~\ref{thm:low energy}]

	Using the expansion for the perturbed resolvents given in \eqref{eq:resolvem}.  The first term is controlled by the
	bounds for the evolution of the free resolvent in
	Theorem~\ref{thm:free high}, specifically \eqref{eqn:freeDirac low}.   Propositions~\ref{prop:A}, \ref{prop:S}, \ref{prop:Erlowlow} control the contribution of the operators $A$, $S-SS_1D_1S_1S$ and $E^\pm(z)$ respectively.  This establishes the theorem in the case when the threshold $\lambda=m$ is regular. If there is a resonance of the first kind, that is an  s-wave  resonance at $\lambda=m$, we bound the additional
	$h_\pm(z)S_1D_1S_1$ term with Proposition~\ref{prop:swave}. 
	\end{proof}

\section{High energy dispersive estimates}\label{sec:high}

We now seek to bound the perturbed Dirac evolution at energies separated from the threshold.    In particular, we show

\begin{prop}\label{prop:high energy}
	Under Assumption~\ref{assumption},
	the following bound holds for any $H^1\times H^1(\R^2) $ atoms $f$ and $g$.
	\be \label{eqn: 2j hig}
		\bigg|\int_0^\infty e^{-it\sqrt{z^2+m^2}}\frac{z}{\sqrt{z^2+m^2}} \chi_j(z)  \left\la [\mathcal R_V^{+}(\lambda)- \mathcal R_V^{-}(\lambda)]f,g\right\ra 
		  dz\bigg|\\ \les   \min(2^{2j},   2^{7j/2} |t|^{-1}).
	\ee
	provided the components of $V$ satisfy the bound
	$|V_{ij}(x)|\les \la x\ra^{-2-}$.

\end{prop}
 
As in the low-energy part of the evolution, we use the Hardy space $H^1$ in place of the Lebesgue space $L^1$. One can prove such bounds with $L^1$, though it requires further iteration of the Born series, which requires more complicated computations and loss of more derivatives on the initial data than presented here.   

The following lemma from \cite{EG1} will be useful to control the spatial integrals that arise in our frequency-localized bounds.
\begin{lemma}\label{EG:Lem}

	Fix $u_1,u_2\in\R^n$ and let $0\leq k,\ell<n$, 
	$\beta>0$, $k+\ell+\beta\geq n$, $k+\ell\neq n$.
	We have
	$$
		\int_{\R^n} \frac{\la z_1\ra^{-\beta-}}
		{|z_1-u_1|^k|z_1-u_2|^\ell}\, dz_1
		\les \left\{\begin{array}{ll}
		(\frac{1}{|u_1-u_2|})^{\max(0,k+\ell-n)}
		& |u_1-u_2|\leq 1,\\
		\big(\frac{1}{|u_1-u_2|}\big)^{\min(k,\ell,
		k+\ell+\beta-n)} & |u_1-u_2|>1.
		\end{array}
		\right.
	$$
\end{lemma}

We begin by employing the resolvent expansion
\begin{align}
    \mathcal R_V^{\pm}(\lambda) =  \mathcal R_0^{\pm}(\lambda)- \mathcal R_0^{\pm}(\lambda) V\mathcal R_0^{\pm}(\lambda)  +\mathcal R_0^{\pm}(\lambda)  V\mathcal R_V^{\pm} (\lambda) V    \mathcal R_0^{\pm}(\lambda).
    \label{born tail}
\end{align}
We already discussed the required bounds for the contribution of the free resolvent in Theorem~\ref{thm:free high}. 
We now consider the contribution of the second term in \eqref{born tail}. 
 Using the estimates in the previous sections, see \eqref{eq:RLRH}, 
\eqref{eqn:R1R2R3}, and the discussion preceding Lemma~\ref{lem:R1h1j}, we have (for $z\gtrsim 1$,  $\lambda=\sqrt{z^2+m^2}$)
\begin{align}\label{eqRL}
\mR_L^\pm(z)(x,y)&= \chi(z|x-y|)\left(\frac{i\alpha\cdot(x-y)}{2\pi|x-y|^2}+\widetilde O_1(z(z|x-y|)^{0-})\right),\\
\mR_H^\pm(z)(x,y)&= e^{\pm i z|x-y|}\widetilde w_\pm(z|x-y|), \nn\\
\frac{\mR_L^+(z)(x,y)-\mR_L^-(z)(x,y)}{\sqrt{z^2+m^2}}&=\widetilde O_1(z/\sqrt{z^2+m^2}).\nn
\end{align}
By symmetry, it suffices to consider the contributions of 
$$\Gamma_1:=(\mR^+_L-\mR^-_L)V\mR_L^+,\,\,\,\,\,\Gamma_2:=\mR^+_LV\mR^+_H,\,\,\,\,\,\,\Gamma_3:=\mR^+_HV\mR^+_H $$
to the Stone's  formula \eqref{eqn: 2j hig}.
\begin{lemma}\label{lem:1highgamma}
	The following bound holds for each $k=1, 2, 3$
	\begin{align}\label{eqn: 2j higGamma}
		\bigg|\int_0^\infty e^{-it\sqrt{z^2+m^2}}\frac{z}{\sqrt{z^2+m^2}} \chi_j(z)\Gamma_k
	  \, dz\bigg| \les  	  \min(2^{2j}, 2^{2j}|t|^{-1/2}, 2^{5j/2} |t|^{-1})
	\end{align}
	provided the components of $V$ satisfy the bound
	$|V_{ij}(x)|\les \la x\ra^{-2-}$.
\end{lemma}
\begin{proof}
For $\Gamma_1$, we need to consider an integral which can be
written as
$$
\bigg| \int_{\R^2} \int_0^\infty e^{-it\sqrt{z^2+m^2}}  \chi_j(z) \chi(z|x-u|)V(u) \chi(z|u-y|) \widetilde O_1(z /|u-y|) 
	  \, dz du  \bigg|.
$$
We can bound the integral by
$$
\min(2^{2j},2^j/t)\int_{\R^2}  \frac{|V(u)|}{|u-y|} du \les \min(2^{2j},2^j/t).
$$
To obtain this we estimated the $z$ integral by ignoring the phase, and by an integration by parts as before. The
$u$ integral is clearly bounded by Lemma~\ref{EG:Lem}.

For $\Gamma_2$ and $\Gamma_3$, we note that direct integration implies the bound $2^{2j}$ as above. To obtain time decay we employ Lemma~\ref{lem:high stat phase2} to
  the oscillatory integral that the phase(s) in
$\mR_H(z)$ provide.  We estimate $\Gamma_3$ only, $\Gamma_2$ is bounded similarly with a smaller power of $2^j$.  With
$\phi_\pm(z)=\sqrt{z^2+m^2}\mp z(|x-u|+|u-y|)/t$,
we consider
$$
	  \int_{\R^2} \int_0^\infty e^{-it\phi_\pm (z)}  \chi_j(z)
	\widetilde O_1(z^2) \widetilde \omega_+(z|x-u|) 
	\widetilde \omega_+(z|u-y|) \, dz\, du.
$$
Define $r:=\max(|x-u|,|u-y|)$ and $s:=\min(|x-u|,|u-y|)$,
we then rewrite the integrand as
$$
	e^{-it\phi_+(z)}\frac{z \chi_j(z)\widetilde \chi (zr)}{(1+zr)^{\f12}} \frac{z^{\f12}}{s^{\f12}}.
$$
Since $r\approx r+s$, and $z^{\f12}\approx 2^{j/2}$, we
apply Lemma~\ref{lem:high stat phase2} to bound the
$z$ integral with
\begin{align*}
	2^{j/2} \min(2^{2j}, 2^{\frac{3j}2}|t|^{-1/2}, 2^{2j} |t|^{-1}) \int_{\R^2} \frac{|V(u)|}{|u-y|^{\f12}}\, du.
\end{align*}
Here, without loss of generality, we took $s=|u-y|$.
The $u$ integral is bounded by the decay of $V$ and
Lemma~\ref{EG:Lem}.  The case of $k=2$ varies only in that
the final integrand is bounded by
$$
	\min(2^{2j}, 2^{\frac{3j}2}|t|^{-1/2}, 2^{2j} |t|^{-1}) \int_{\R^2} \frac{|V(u)|}{|u-y|}\, du.  \qedhere
$$
\end{proof}

The following lemma finishes the proof of Proposition~\ref{prop:high energy}.
\begin{lemma}\label{lem:2highgamma}
	The following bound holds for any $H^1\times H^1(\R^2)$ atoms $f$ and $g$.
	\be \label{eqn: 2j hightail}
		\bigg|\int_0^\infty e^{-it\sqrt{z^2+m^2}}\frac{z\chi_j(z)}{\sqrt{z^2+m^2}} \left\la  \mathcal R_0^{\pm}(\lambda)V\mathcal R_V^{\pm} (\lambda) V    \mathcal R_0^{\pm}(\lambda) f, g \right\ra  
	  \, dz\bigg|    \les  	  \min(2^{2j},   2^{7j/2} |t|^{-1}).
	\ee
	provided the components of $V$ satisfy the bound
	$|V_{ij}(x)|\les \la x\ra^{-2-}$.
\end{lemma}
\begin{proof}
In this proof we consider only the case $t>0$, and the '$+$' terms, and drop the superscripts. By symmetry, it suffices to consider the contributions of the following to \eqref{eqn: 2j hightail}:
$$
\Gamma_1:=\mR_L V\mR_VV\mR_L,\,\,\,\,\,\Gamma_2:=\mR_L V\mR_VV\mR_H,\,\,\,\,\,\Gamma_3:=\mR_H V\mR_VV\mR_H.
$$
Consider the contribution of $\Gamma_1$. 
We rewrite $\mR_L$, see \eqref{eqRL}, as follows:
\begin{align*}
\mR_L(z)(x,y) &=  \frac{i\alpha\cdot(x-y)}{2\pi|x-y|^2}+\widetilde{\chi}(z|x-y|) \widetilde  O(|x-y|^{-1}) +\chi(z|x-y|)\widetilde O_1(z(z|x-y|)^{0-})\\
&=\frac{i\alpha\cdot(x-y)}{2\pi|x-y|^2}+ \widetilde O_1(z^{1/2}|x-y|^{-1/2})=:\frac{i\alpha\cdot(x-y)}{2\pi|x-y|^2}+\mR_{L1}.
\end{align*}
Using Theorem~1 in \cite{dinglu}, the first summand above maps $H^1$ to $L^2$.
Thus, we can write
$$
\la \Gamma_1 f,g\ra= \la  \mR_V V\widetilde f, V\widetilde g\ra +  \la \mR_V V  \mR_{L1} f,V\widetilde g\ra + \la  V \widetilde f, \mR_V V  \mR_{L1}  g\ra +
\la \mR_{L1} V\mR_V V \mR_{L1} f,    g\ra,
$$
where $\widetilde{f}, \widetilde g \in L^2\times L^2$. 
Therefore, by limiting absorption principle we have:
\begin{multline*}
 |\la \Gamma_1 f,g\ra| \les \|V\widetilde g \|_{L^2_\sigma}  \|V\widetilde f \|_{L^2_\sigma} + z^{1/2} \|V\widetilde g \|_{L^2_\sigma} \|f\|_{L^1} \sup_u \left\|\frac{V(\cdot)}{|\cdot-u|^{1/2}}\right\|_{L^2_\sigma}\\ + z^{1/2} \|V\widetilde f\|_{L^2_\sigma} \|g\|_{L^1} \sup_u \left\|\frac{V(\cdot)}{|\cdot-u|^{1/2}}\right\|_{L^2_\sigma}+ z \|f\|_{L^1}\|g\|_{L^1} \sup_u \left\|\frac{V(\cdot)}{|\cdot-u|^{1/2}}\right\|_{L^2_\sigma}^2   \les z.
\end{multline*}
Using this bound we estimate the contribution of $\Gamma_1$ to \eqref{eqn: 2j hightail} by $2^{2j}$. The same bound holds for $\Gamma_2$ and $\Gamma_3$. 
We also have 
$$
|\partial_z\la \Gamma_1 f,g\ra|    \les z,
$$
since the worst terms are the ones when the derivative hits $\mR_V$. Using this bound after an integration by parts we   estimate the contribution of $\Gamma_1$ to \eqref{eqn: 2j hightail} by $2^{2j}/t$. 

For $\Gamma_3$,  it suffices to estimate   
$$
	\sup_{x,y}\bigg|\int_0^\infty e^{-i 2^{-3j} t \phi(z)}   a(z,x,y) dz\bigg|, $$
where
$$
\phi(z)= 2^{3j}\left(\sqrt{z^2+m^2}- z(|x|+|y|)/t\right),
$$
and (with $r=|x-u_1|$, $s=|u_1-y|$)
$$
a(z,x,y)=  \int_{\R^4} \chi_j(z)
	\widetilde O_1(z^2) e^{iz(r-|x|)} \widetilde \omega(zr) [V\mR_VV](u_1,u_2)
	e^{iz(s-|y|)} \widetilde \omega(zs)  \,  du_1\,d u_2.
$$
Note that in the support of $a$, we have $\phi^{\prime\prime} \approx 1$. Also note that, using Lemma~\ref{EG:Lem} and the limiting absorption principle, we have
$$
|a(z,x,y)|+|\partial_z a(z,x,y)|\les 2^j \chi_j(z) \la x\ra^{-1/2} \la y\ra^{-1/2}.
$$
Therefore by Lemma~\ref{stat phase}, we estimate the integral above by
$$
\int_{|z-z_0|<\sqrt{2^{3j}/t }} |a(z)|\, dz 
    		+t^{-1} 2^{3j} \int_{|z-z_0|>\sqrt{2^{3j}/t }} \bigg( \frac{|a(z)|}{|z-z_0|^2}+
    		\frac{|a'(z)|}{|z-z_0|}\bigg)\, dz,
$$
where $z_0=m\frac{|x|+|y|}{\sqrt{t^2-(|x|+|y|)^2}}$. In the case when $z_0$ is in a small neighborhood of the support of $a$ we must have  $t\approx |x|+|y|$. Therefore, in this case, we have the bound
$$
2^j \la x\ra^{-1/2} \la y\ra^{-1/2} \left(\sqrt{2^{3j}/t }+t^{-1} 2^{3j} \frac{2^j}{\sqrt{2^{3j}/t }}\right)
\les 2^{7j/2}/t.
$$
In the case $t\not \approx |x|+|y|$, we have 
$$
\left|\partial_z\left(\sqrt{z^2+m^2}- z(|x|+|y|)/t\right)\right| \gtrsim 1.
$$
An integration by parts together with the bounds on $a$ imply that the integral is bounded by
$2^{2j}/t$. The proof for $\Gamma_2$ is similar to the cases considered above.
\end{proof}

\section{Classification of threshold resonances}\label{sec:resonanceclass}

In this section we provide a full characterization of threshold obstructions.  We classify distributional solutions to $H\psi =m\psi$ and relate them to the spectral subspaces and terms that arise in the inversion of the operators $M^\pm (z)=U+v\mR_0^\pm \left(\sqrt{m^2+z^2}\right)v^*, \,\,\,\,0<z\ll 1.$   We start with the absolute boundedness of the operator $QD_0Q$ in the case $S_1=0$. In the case $S_1\neq 0$, the proof is similar.
\begin{lemma}\label{lem:QDQ abs}

	If $|v_{ij}(x)|\les \la x\ra^{-1-}$ then 
	$QD_0Q$ is an absolutely bounded operator.

\end{lemma}

\begin{proof}
This is similar to the proof Lemma 8 in \cite{Sc2}.
	Assume that $0\neq f\in Q(L^2\times L^2)$ with
	$QUf=0$.  Then $Uf\in  \, \text{span}(a,c)^T$, which can
	be expressed as $f=CUv(1,0)^T$ for some $C\neq 0$.
	Then the assumption $Pf=M_{11}v^*f=0$ and the identity $V=v^*Uv$ imply that
	\begin{align*}
		0= M_{11}v^*Uv(1,0)^T= M_{11}V(1,0)^T  = \left(\int_{\R^2} V_{11}(y)\, dy \right)
		(1,0)^T,
	\end{align*}
	where $V_{11}$ is the top left component of the matrix
	potential $V$. Since this argument can be reversed, we have shown that
	\begin{align*}
		\textrm{ker}_{Q(L^2\times L^2)}(QUQ)=\{0\}\qquad
		\textrm{if and only if} \qquad
		\int_{\R^2}V_{11}(y)\, dy\neq 0.
	\end{align*}
	Moreover, if $\int_{\R^2}V_{11}(y)\, dy= 0$, then the kernel is the span of the vector $Uv(1,0)^T$.
	Also note, using  $V=v^*Uv$, that 
	$$
		V_{11}(y)=\textrm{sign}(\lambda_1)|a|^2(y)+
		\textrm{sign}(\lambda_2)|c|^2(y).
	$$
  
	We consider first the case in which $\int_{\R^2}V_{11}(y) \, dy\neq 0$.  In this case,
	$QUQ$ is an invertible operator on $Q(L^2\times L^2)$. To see this,  
	for any $g\in  L^2\times L^2 $ with $Qg=g$,  
	define
	$$
		f=Ag:=Ug+c_0Uv(1,0)^T \qquad \textrm{with}
		\qquad c_0=-\frac{\la Ug,v (1,0)^T\ra}{\int_{\R^2} V_{11}(y)\, dy}.
	$$
	Note that
	$$Pf= P(Ug)+c_0P(Uv(1,0)^T)=\frac{(a,c)^T}{\|(a,c)\|_2^2} \Big(\la Ug,v(1,0)^T\ra +c_0   \int_{\R^2}
		V_{11}(y)\, dy\Big)=0, 
	$$
	by the definition of $c_0$. Therefore  $Qf=f$. Also note that 
	$$
	QUQf=QUf=QU(Ug+c_0 Uv(1,0)^T) = g+c_0Qv(1,0)^T=g,
	$$
	since $Qv(1,0)^T=0$. Therefore, the operator $A $ is the inverse of $QUQ$, and it is easy to see using the explicit formula that $A$ is absolutely bounded on $Q(L^2\times L^2)$.
    
    Using resolvent identity twice, we can write 
    $$
    D_0=(Q(U+v\mG_0 v^*)Q)^{-1}=  A-A(v\mathcal G_0 v^*)A\\
		+D_0 (v\mathcal G_0 v^*)A
		(v\mathcal G_0 v^*)A.	
    $$
	The first term has already been shown to be
	absolutely bounded.  The second term, recalling
	\eqref{def:mG0},  is the sum of a
	Hilbert-Schmidt operator, $2mvG_0 I_1v^*$ and an operator $-iv\alpha \cdot \nabla G_0 v^*$, which 
	is bounded in absolute value by $|v|\, \mathcal I_1\, |v^*|$ with $\mathcal I_1$ a fractional integral operator.  Recall that $\mathcal I_1: 
	L^{2,\sigma}\to L^{2,-\sigma}$ for $\sigma>1$, see e.g. Lemma~2.3 of \cite{Jen}. Therefore the second term is also absolutely bounded. Since the composition of a bounded operator with an Hilbert-Schmidt operator is Hilbert-Schmidt, and hence absolutely bounded, for the third term it suffices to prove that $(v\mathcal G_0 v^*)A
		(v\mathcal G_0 v^*)$ is Hilbert-Schmidt. This follows from the explicit formula for $A$,  the inequalities 
 \begin{align*}
 &\int_{\R^2} \frac{\la w\ra^{-2-}}{|x-w||w-y|} d w    \les 1+  \log^-|x-y|, \\ 
 &\int_{\R^4} \la x\ra^{-2-} (1+  \log^-|x-y|)^2 \la y\ra^{-2-} dx dy   <\infty, 
 \end{align*}
and similar inequalities involving less singular integrands.  

We now consider the case in which $\int_{\R^2}V_{11}(y) \, dy = 0$. In this case $0$ is an isolated point of the spectrum of $QUQ$ whose essential spectrum is contained in $\{\pm 1\}$. Let $\pi_0$ be the Riesz projection onto the kernel of $QUQ$. By the calculation in the beginning of the proof, we have
$$
\pi_0(f)=\frac{\la f,Uv(1,0)^T\ra}{\|(a,c)^T\|_2^2} Uv(1,0)^T.
$$
Similarly, for $g\in Q(L^2\times L^2)$,
$$
 f=Ag:=Ug+c_1\big(v(1,0)^T-Uv(1,0)^T\big),\,\,\,c_1=-\frac{\la Ug,v(1,0)^T\ra}{\|(a,c)^T\|_2^2},
$$
satisfies $Qf=f$ and $(QUQ+\pi_0)f =g$. By the explicit formula, $A=(QUQ+\pi_0)^{-1}$ is absolutely bounded. The previous argument yields the claim by writing 
$$
D_0=(Q(U+v\mG_0 v^*)Q)^{-1}=  A-A(v\mathcal G_0 v^*-\pi_0)A\\
		+D_0 (v\mathcal G_0 v^*-\pi_0)A
		(v\mathcal G_0 v^*-\pi_0)A.	\qedhere
$$

	 
\end{proof}

\begin{lemma}\label{lem: S1}
	
	If $|v_{ij}(x)|\les \la x\ra^{-1-}$ and $\phi \in S_1(L^2\times L^2) =  \text{Ker} \, (QTQ)$ then $\phi=Uv\psi$ where 
	$$
		\psi = c_0 (1,0)^T -\mathcal G_0 v^* \phi, \qquad
		c_0=\frac{\la (a,c)^T,T\phi\ra }{\| (a,c)^T\|_2^2}.
	$$
	Moreover, $\psi\in L^\infty \times L^\infty$, and it is a distributional solution of $(H-mI)\psi=0$.
\end{lemma}

\begin{proof}
	
	Assume that $\phi \in \text{Ker} \, (QTQ)$. Then  
	$Q\phi=\phi$ and $PT\phi=c_0v(1,0)^T$ by definition of the projection $P$. We have
	\begin{align*}
		0&=QTQ\phi=(1-P)TQ\phi=T\phi-PT\phi=U\phi+v\mathcal G_0 v^*\phi-c_0v(1,0)^T.
	\end{align*}
	Multiplying
	the above expression by $U$ on the left, we arrive at
	\begin{align*}
		\phi&=-Uv\mathcal G_0 v^* \phi +c_0Uv(1,0)^T:=Uv \psi,
	\end{align*}
	where
	$$
	\psi=-\mathcal G_0 v^* \phi +c_0(1,0)^T.
	$$
	We now prove that $(H-mI)\psi=0$. Noting that $(D_m-mI) (1,0)^T=0$, we have
	\begin{multline}\label{eqn:mresonance}
		(H-mI)\psi=(D_m-mI)\psi +V\psi =(D_m-mI)\psi +v^*Uv\psi \\
		= (D_m-mI)(-\mathcal G_0v^* \phi+c_0 (1,0)^T) + v^*\phi 
		= - (D_m-mI) \mathcal G_0v^* \phi  + v^*\phi.
	\end{multline}
	Therefore it suffices to prove that $(D_m-mI) \mathcal G_0v^* \phi=v^*\phi$. We compute using the definition of $\mG_0$, and the identities $(-i\alpha \cdot \nabla)^2=-\Delta$ and $I_2I_1=0$ that 
	\begin{multline}\label{eqn:mresonance1}
		(D_m-mI) \mathcal G_0v^* \phi =(-i\alpha \cdot \nabla-2mI_2) (-i\alpha \cdot \nabla G_0 +2m G_0 I_1) v^* \phi 
		\\= (-\Delta G_0 -2mi \alpha \cdot \nabla G_0 I_1 +2mi  I_2 \alpha \cdot \nabla G_0) v^*\phi \\
		=  v^* \phi - 2mi (\alpha I_1-  I_2 \alpha) \cdot \nabla G_0   v^*\phi = v^*\phi.
	\end{multline}
	In the third equality we used  $G_0=(-\Delta)^{-1}$, and $|v_{ij}(x)|\les \la x\ra^{-1-}$, and the last equality follows from  $\alpha_jI_1=I_2\alpha_j$, $j=1,2.$ This proves that $(H-mI)\psi=0$.
	
	We now prove that $\psi$ is bounded.  Writing
	\begin{align}\label{def:psi}
		\psi=-\mathcal G_0 v^* \phi +c_0(1,0)^T= i\alpha \cdot \nabla G_0 v^* \phi -2mG_0 I_1v^* \phi +c_0(1,0)^T,
	\end{align}
	we only need to show that the first and second summands are in $ L^\infty \times L^\infty$. Consider the second term. The boundedness is clear
	on $B(0,4)$.
	Then, using $M_{11} v^* \phi =0 $ (see part i) of Remark~\ref{rmk:resonances}),
	we can write
	\begin{multline} \label{eq:G0trick}
		 [G_0 I_1 v^* \phi](x) = G_0 I_1v^* \phi +\frac1{2\pi} \log|x| M_{11} v^*\phi \\ =-\frac{1}{2\pi} I_1 \int_{\R^2}
		[\log|x-y|-\log |x|] (v^*\phi)(y)\, dy.
	\end{multline}
	The bound follows by using the inequality 
	\begin{align}	\label{logtrick}
		\Big|\log\Big(\frac{|x-y|}{|x|}\Big)\big|\les  1+\log(\la y\ra)+\log^-(|x-y|),\,\,\,\,|x|>4,
	\end{align}
	and  the bound  $|v_{ij}(x)|\lesssim \la x\ra^{-1-}$.
	
	To see the boundedness of the first term on the right hand side of \eqref{def:psi}, recall that
	$$
	\phi =-Uv\mathcal G_0 v^* \phi +c_0Uv(1,0)^T= iUv \alpha \cdot \nabla G_0v^* \phi -2m UvG_0 I_1 v^* \phi+c_0Uv(1,0)^T.
  $$
Note that if $f\in L^2(\R^2)$,   for any $p\in (1,2)$ and 
 $q\in(2,\infty)$, $\frac1q=\frac1p-\frac12,$ we have 
$$
\Big\|\int \frac1{|x-y|}\la y\ra^{-1-} |f(y)|dy \Big\|_{L^q} \les \|\mathcal I_1\|_{L^p\to L^q} \|\la \cdot \ra^{-1-}f\|_{L^p}\les 
\|\la \cdot \ra^{-1-}\|_{L^q} \|f\|_{L^2}\les \|f\|_{L^2}
$$  
by Lemma~\ref{EG:Lem}, the $L^{p}\to L^q$ boundedness of the fractional integral operator $\mathcal I_1$ in $\R^2$ and H\"older's inequality (since  $\frac1p=\frac1q+\frac12$ and $q>2$).  This implies using the bound on $v$ that the first summand in the definition of $\phi$ is in $L^{\infty-} \times L^{\infty-} $. 
By $L^{\infty-}$ we mean $L^q$ for arbitrary large, but finite, $q$.
The second summand is bounded by the argument above. Therefore
$\phi\in (L^{\infty-}\times L^{\infty-}) \cap (L^2\times L^2)$. The boundedness of the first term in the definition of $\psi$ follows from this using that
\begin{multline*} 
\int \frac1{|x-y|}\la y\ra^{-1-} |f(y)|dy \leq \\\|f\|_{L^{\infty-}} \Big[\int_{|x-y|<1} \frac1{|x-y|^{1+}} dy\Big]^{1-} + \|f\|_{L^2} \Big[ \int_{|x-y|>1}  \la  y\ra^{-2-} dy\Big]^{1/2}\\
\les \|f\|_{L^{\infty-}}   + \|f\|_{L^2}.  \qedhere
	\end{multline*}
\end{proof}

The following lemma provides more detailed information on $S_1$, however it requires more decay from the potential $V$. 
\begin{lemma}\label{lem:moreonS1} Assume that $|v_{ij}(x)|\les \la x\ra^{-2-}$. Let $\phi=Uv\psi \in S_1(L^2\times L^2)$. We have 
$$
		\psi = c_0 (1,0)^T +\Gamma_1+ \Gamma_2,		 
$$
where 
$$
\Gamma_1=-\frac{mx}{\pi\la x\ra^2}\cdot  \int_{\R^2}  y I_1 v^*(y)\phi(y)\, dy - \frac{i}{2\pi}\alpha \cdot \frac{x}{\la x \ra^2} M_{22} v^* \phi,
$$
and $\Gamma_2\in L^p\times L^p, $ for any $p\in [2,\infty]$. 
In particular, $\psi- c_0(1,0)^T\in L^p\times L^p$ for any $2<p\leq \infty$.
	
\end{lemma}

\begin{proof}
	
Recall from Lemma~\ref{lem: S1} that $\psi = c_0 (1,0)^T -\mathcal G_0 v^* \phi $. Therefore, we define
\be\label{eq:G1+G2}
\Gamma_1+ \Gamma_2:=  -\mathcal G_0 v^* \phi =  i\alpha \cdot \nabla G_0v^* \phi-2m G_0 I_1v^* \phi.
\ee
Below we analyze  the right hand side of \eqref{eq:G1+G2}; the combination of the non-$L^2$ pieces gives $\Gamma_1$, the remaining $L^2$ pieces give $\Gamma_2$.

	We already know from Lemma~\ref{lem: S1} that $\psi \in L^\infty\times L^\infty$. Therefore it suffices to prove that $\Gamma_2\in L^2\times L^2$ on  $S:=\{ x\in \R^2: |x|>10\}$.
	We start with the second summand. We use \eqref{eq:G0trick} to write
\begin{multline*}
-4\pi G_0I_1 v^*\phi(x)  = \int \log\Big(\frac{|x-y|^2}{|x|^2}\Big)I_1 v^*\phi(y) dy\\
 =\int_{A}\log\Big(\frac{|x-y|^2}{|x|^2}\Big) I_1 v^*\phi(y)\, dy +\int_{B}\log\Big(\frac{|x-y|^2}{|x|^2}\Big) I_1 v^*\phi(y)\, dy.
\end{multline*}
Here  $A:=\{y\in\R^2:|y|<|x|/10\}$, $B:=\R^2\setminus A$. 
We note that, on the set $A$, $\big||y|^2-2x\cdot y\big|/|x|^2<\frac{1}{2}$, and hence
\begin{align*}
	\log\bigg(\frac{|x-y|^2}{|x|^2}\bigg) =\ln\bigg(1+\frac{|y|^2}{|x|^2}-\frac{2x\cdot y}{|x|^2}\bigg) 
=-\frac{2x\cdot y}{|x|^2}+O\bigg( \frac{\la y\ra^{1+}}{\la x\ra^{1+}}\bigg).
\end{align*}
Therefore, also using that $|v_{ij}(y)|\les \la y\ra^{-2-}$, we have 
\begin{multline*}
\int_{A} \log\bigg(\frac{|x-y|^2}{|x|^2}\bigg)I_1 v^*\phi\, dy
=-\frac{2x}{|x|^2}\cdot \int_{A}  y I_1 v^*(y)\phi(y)\, dy
	+O\bigg( \frac{\int_{A} \la y\ra^{1+} |v^*(y)\phi(y)|\, dy}{\la x\ra^{1+}} \bigg)
	\\=-\frac{2x}{|x|^2}\cdot \int_{A}  y I_1 v^*(y)\phi(y)\, dy
	+O_{L^2(S)}(1).
\end{multline*}
Note that (for $x\in S$)
\begin{align}\nn
\frac{2x}{|x|^2}\cdot \int_{B}  y I_1 v^*(y)\phi(y)\, dy &=O\big(\la x \ra^{-1} \|\la y \ra^{-1-}\|_{L^2(B)}\big)=O\big(\la x \ra^{-1-} \big), \\
\label{xjapx}\frac{x}{|x|^2}-\frac{x}{\la x\ra^2} &=O\big(\la x \ra^{-3} \big).
\end{align}
Therefore 
\begin{align*}
\int_{A} \log\bigg(\frac{|x-y|^2}{|x|^2}\bigg)I_1 v^*\phi\, dy = -\frac{2x}{\la x\ra^2}\cdot \int_{\R^2}  y I_1 v^*(y)\phi(y)\, dy
	+O_{L^2(S)}(1).
\end{align*}
We have 
$$\Big|\log\bigg(\frac{|x-y|^2}{|x|^2}\bigg)\Big| \les 1+\log(\la y\ra)+\log^-(|x-y|),$$
provided that $x\in S$, $y\in B$. Therefore,
\begin{multline*}
	\Big|\int_{ B} \log\bigg(\frac{|x-y|^2}{|x|^2}\bigg)I_1v^*(y)\phi(y)\, dy\Big|
	 \les  \\ \frac{1}{\la x\ra^{1+}} \int_{  B} \la y\ra^{1+} (1+|y|^{0+}+|x-y|^{0-})|v^*(y)\phi(y)|\, dy  =O_{L^2(S)}(1).  
\end{multline*}
This prove that
$$
-2m G_0 I_1v^* \phi=-\frac{mx}{\pi\la x\ra^2}\cdot  \int_{\R^2}  y I_1 v^*(y)\phi(y)\, dy
	+O_{L^2(S)}(1).
$$
Now we consider the first summand. We have
\begin{multline*}
i\alpha \cdot \nabla G_0v^* \phi=-\frac{i}{2\pi}\int_{\R^2}\alpha\cdot \frac{x-y}{|x-y|^2} v^*(y) \phi(y) dy \\
=-\frac{i}{2\pi}\int_{\R^2}\alpha\cdot \Big[\frac{x-y}{|x-y|^2}-\frac{x}{|x|^2}\Big] v^*(y) \phi(y) dy - \frac{i}{2\pi}\alpha\cdot \frac{x}{|x|^2} M_{22} v^* \phi,
\end{multline*}
since $M_{11}v^*\phi=0$. Therefore, the following claim and \eqref{xjapx} finishes the proof of the lemma. 

Claim:
\be\label{eq:tempclaim}
\int_{\R^2} \Big|\frac{x-y}{|x-y|^2}-\frac{x}{|x|^2}\Big| |v^*(y) \phi(y)| dy =O_{L^2(S)}(1).
\ee
To prove this claim first note that (for $x\in S$)
$$
\Big|\frac{x-y}{|x-y|^2}-\frac{x}{|x|^2}\Big| \les \left\{\begin{array}{ll} \frac{\la y\ra^{0+}}{\la x\ra^{1+}}  & y\in A, \\
 \frac{\la y\ra^{0+}}{\la x\ra^{0+} |x-y|} +  \frac{\la y\ra^{0+}}{\la x\ra^{1+}  } & y\in B. \end{array}\right.
$$
The contribution of the nonsingular terms is in $L^2$ as above. Therefore,
$$
\eqref{eq:tempclaim} = \la x\ra^{0-} \int_{B}  \frac{\la y\ra^{0+}}{|x-y|} |v^*(y) \phi(y)| dy + O_{L^2(S)}(1).
$$
The integral is in $L^p$ for any $p>2$ because of the boundedness of the fractional integral operator $\mathcal I_1$ as in the proof of Lemma~\ref{lem: S1}. The claim now follows from  H\"older's inequality since 
$\la x\ra^{0-}\in L^q$ for some $q<\infty$.  
\end{proof}

\begin{rmk}\label{rmk:m swave}
	i) We note that  
	there is a threshold s-wave resonance at $\lambda=m$
	for the free Dirac equation (when $V=0$)
	as the constant
	function $\psi=(1,0)^T$ solves $(D_m-mI)\psi=0$.  \\
	ii)	One can perform a similar analysis centered 
	near $\lambda=-m$ with $\lambda=-\sqrt{m^2+z^2}$.  In this
	case the free equation has a threshold
	resonance at $\lambda=-m$ as $\psi=(0,1)^T$ solves
	$(D_m+mI)\psi=0$.
 \end{rmk}

\begin{lemma}\label{lem:S11}
	
	Let  $|v_{ij}(x)|\les \la x\ra^{-2-}$. If $(H-mI)\psi=0$ for some 
	$\psi=c (1,0)^T+\Gamma_1+\Gamma_2$ with $\Gamma_1\in L^p\times L^p$ for some
	$2<p<\infty$ and $\Gamma_2\in L^2\times L^2$, then $\phi:=Uv\psi  \in S_1(L^2\times L^2)$. Moreover, 
	$\psi = c (1,0)^T -\mathcal G_0 v^* \phi$ and $c=c_0$ as in Lemma~\ref{lem: S1}. Furthermore,  $\Gamma_1,\Gamma_2\in 
	L^p\times L^p$ for all $p\in(2,\infty]$. 
	
\end{lemma}

\begin{proof} First of all, using $\phi=Uv\psi$ and the assumption on $v$, $\psi$, we conclude that  $\phi\in L^2\times L^2$  and $v^*\phi\in L^1\times L^1$. We have, since $(H-mI)\psi=(D_m-mI+V)\psi=0$,
$$
		v^* \phi = v^*Uv\psi=V\psi = -(D_m-mI) \psi.
$$
We note that $\phi \in Q(L^2\times L^2)$ if $P\phi =0$, which is equivalent to $  M_{11}v^*\phi=\int I_1 v^*\phi= 0$.   
Using the identity above we have 
$$
(-i\alpha \cdot \nabla +2m I_1)v^*\phi=(D_m+mI)v^*\phi = - (D_m+mI)(D_m-mI) \psi=\Delta \psi.
$$ 
Therefore, it suffices to prove that
$$
\int \alpha \cdot \nabla  (v^*\phi) =0,\,\,\,\,\text{ and } \int \Delta \psi =0.
$$
Both of these follow easily using $v^*\phi \in L^1\times L^1$, and the assumptions on $\Gamma_1, \Gamma_2$, see e.g. \cite[Lemma 6.4]{JN}. 
Thus,  $\phi \in Q(L^2\times L^2)$.
	
	We now claim that $\psi=c(1,0)^T-\mathcal G_0 v^* \phi$.  To show this,  compute
	\begin{multline*}
		(D_m-mI)(\psi+\mathcal G_0 v^* \phi) =(D_m-mI)\psi+(D_m-mI)\mathcal G_0 v^* \phi \\ =
		-V\psi+v^*\phi 
		 =-v^*\phi+v^*\phi =0.
	\end{multline*}
	In the second equality we used \eqref{eqn:mresonance1}. If we apply $(D_m+mI)$ to this equality, we obtain
	$$
	-\Delta (\psi+\mathcal G_0 v^* \phi )=0,
	$$
which implies that   $\psi+\mathcal G_0 v^* \phi  =(c_1,c_2)^T$, since $  \psi+\mathcal G_0 v^* \phi \in (L^2\times L^2) + (L^\infty\times L^\infty)$ (see the proof of Lemma~\ref{lem: S1}). Finally, since 
$$0=(D_m-mI)(\psi+\mathcal G_0 v^* \phi) = (D_m-mI)(c_1,c_2)^T=-2m(0,c_2)^T, $$
 we conclude that $c_2=0$. Therefore, we
	may write $\psi=c(1,0)^T-\mathcal G_0v^*\phi$.  Now, using that $Q\phi=\phi$ and
	the representation of $\psi$, we consider
	\begin{align*}
		TQ\phi&=(U+v\mathcal G_0 v^*) \phi=U\phi+v(c(1,0)^T-\psi) =v\psi+cv(1,0)^T-v\psi =cv(1,0)^T.
	\end{align*}
	Therefore, 
	$$
		QTQ\phi= cQv(1,0)^T=0,\,\,\,\text{ and } PT\phi =cPv(1,0)^T,
	$$
which implies that $\phi\in S_1$ and $c=c_0$. This finishes the proof together with Lemma~\ref{lem:moreonS1}. 	
\end{proof}

\begin{lemma}\label{lem:S2} Assume that  $|v_{ij}(x)|\les \la x\ra^{-2-}$.  Fix $\phi =Uv\psi \in S_1(L^2\times L^2)$. Then,
	$\phi \in S_2(L^2\times L^2)$  if and only if
	$\psi \in  L^p\times L^p $ for all $p\in (2,\infty]$, that is $c_0=0$.
	
\end{lemma}

\begin{proof}
	
	Note that $\phi \in S_2$ means $S_1TPTS_1\phi=0$, which holds if and only if 
	\begin{align}
		0=\la \phi, S_1TPTS_1\phi\ra=\la PT \phi, PT \phi \ra =\|PT  \phi\|_2^2.
	\end{align}
	Therefore, $\phi \in S_2$ if and only if $c_0=0$.  Finally note that   the representation in  Lemma~\ref{lem:moreonS1} implies that
	$c_0=0$ if and only if
	$\psi \in  L^p\times L^p $ for all $p\in (2,\infty]$. 
\end{proof}

\begin{rmk}\label{rmk:ranks} i) By the representation in Lemma~\ref{lem:moreonS1},  if $\phi_1, \phi_2\in S_1$, then there is a constant $c$ so that
$\phi_1-c\phi_2 $ or $\phi_2-c\phi_1 \in L^p\times L^p$, $p\in(2,\infty]$. Therefore in $S_2$. This implies that the rank of $S_1-S_2$ is at most $1$. \\
ii) Note that $\Gamma_1$ in Lemma~\ref{lem:moreonS1} can be written as
$$
\Gamma_1= w_1\frac{x_1}{\la x\ra^2} +w_2 \frac{x_2}{\la x\ra^2}, 
$$
where the constant vectors $w_j$ are defined as
$$
w_j=-\frac{m}\pi \int_{\R^2}y_j I_1 v^*(y)\phi(y) dy-\frac{i}{2\pi} \alpha_j M_{22}v^*\phi,  \,\,\,j=1,2.
$$
Also note that $I_2w_j=0$, $j=1,2$, since $I_2I_1=I_1M_{22}=0$ and $I_2\alpha_j=\alpha_j I_1$.\\
ii) Below we   prove that $\phi \in S_3$ if and only if $\psi \in L^2\times L^2$.  This and part i) imply that the rank of $S_2-S_3$ is at most $2$.
\end{rmk}

\begin{lemma}\label{G2 invert}
If $|v_{ij}|\les \la x\ra^{-3-}$ then  the operator $S_3v\mG_2v^*S_3$ on $S_3L^2$ is invertible. Furthermore, for $f\in S_3L^2$, we have
\begin{align}\label{G0 to G2 ident}
	\la \mG_2 v^*f, v^*f\ra =\frac{1}{2m} \la \mG_0v^*f,\mG_0v^*f\ra.
\end{align}
\end{lemma}

\begin{proof}
Noting that $S_3v\mG_2v^*S_3$ is an Hilbert Schmidt operator, it suffices to check that the kernel is empty. Given $f$ in the kernel of $S_3v\mG_2v^*S_3$, since $S_3\leq S_2 \leq S_1\leq Q$,  we have
\be \label{tempor}
M_{11}v^*f=0, \text{ and } S_2v\mG_1v^*S_2 f=0.
\ee
Using Lemma~\ref{lem:R0exp} for $\mR_0(\lambda)$, $\lambda=\sqrt{m^2+z^2}$, and using the bound on $|v_{ij}|$ we have  
\be\label{eq:wid}
\frac1{z^2}\left\la \big[\mR_0(\lambda)-\mG_0\big]v^*f,v^*f \right\ra =\left\la  \mG_2v^*f,v^*f \right\ra+o(1) =o(1)
\ee
as $z\to 0$.
Using this with $z=iw$, $0<w \ll m$, we calculate the left hand side of \eqref{eq:wid} 
$$
\int_{\R^2}\left\la A(w,\xi)\widehat{v^*f}(\xi), \widehat{v^*f}(\xi)\right\ra_{\C^2} d\xi,
$$
where the Fourier multiplier $A(w,\xi)$ is given by 
\begin{multline*}
A(w,\xi)=-\frac1{w^2}\Bigg[\frac1{w^2+|\xi|^2} \left(\begin{array}{cc}m+\sqrt{m^2-w^2} & \overline\xi\\ \xi & \sqrt{m^2-w^2} -m\end{array}\right)-\frac1{|\xi|^2}\left(\begin{array}{cc}2m & \overline\xi\\ \xi & 0\end{array}\right) \Bigg]\\
= \frac{1}{(w^2+|\xi|^2)|\xi|^2} \left(\begin{array}{cc}2m+\frac{|\xi|^2}{w^2} (m-\sqrt{m^2-w^2}) & \overline\xi\\ \xi & \frac{|\xi|^2}{w^2} (m-\sqrt{m^2-w^2}) \end{array}\right).
\end{multline*}
Here $\xi=(\xi_1,\xi_2)$ is identified with $\xi_1+i\xi_2$. Note that, for $\xi\neq 0$, $A$ is positive definite, self-adjoint, and its eigenvalues are 
$$
\lambda_{1,2}=\frac{1}{(w^2+|\xi|^2)|\xi|^2} \Big(m+\frac{|\xi|^2}{w^2}(m-\sqrt{m^2-w^2})\pm\sqrt{m^2+\xi^2}\Big).
$$
It is straightforward to check that $\lambda_{1,2}$ are nonincreasing functions of $w\in(0,m)$.
Therefore using monotone convergence theorem after diagonalizing, we have
$$
0=\lim_{w\to 0^+}\int_{\R^2}\left\la A(w,\xi)\widehat{v^*f}(\xi), \widehat{v^*f}(\xi)\right\ra d\xi=\int_{\R^2} \left\la A(0,\xi)\widehat{v^*f}(\xi), \widehat{v^*f}(\xi)\right\ra_{\C^2} d\xi,
$$
where 
$$
A(0,\xi)
= \frac{1}{ |\xi|^4 } \left(\begin{array}{cc}2m+\frac{|\xi|^2}{2m}  & \overline\xi\\ \xi & \frac{|\xi|^2}{2m}   \end{array}\right).
$$
Since $A(0,\xi)$ is also positive definite and self adjoint, we conclude that $\widehat{v^* f}(\xi)=0$. This implies  that  $v^*f=0$ since  $v^*f$ has $L^1$ entries. Recalling the definition of $v$, we obtain    $\eta_1f_1=0, \eta_2f_2=0$. Also noting that $f=Uv\psi=(\eta_1 h_1,\eta_2h_2),$ where $h=UB\psi$, we conclude that $f=0$. 
Therefore $S_3 v\mG_2v^*S_3$ is invertible.
 
Further, noting that
$$
A(0,\xi)= \frac1{2m}\Bigg[\frac{1}{ |\xi|^2 } \left(\begin{array}{cc}2m  & \overline\xi\\ \xi & 0  \end{array}\right)\Bigg]^2,
$$
we obtain the identity \eqref{G0 to G2 ident} for any $f\in S_3L^2$.
\end{proof}

\begin{lemma}\label{S3 characterization}
Assume that  $v(x)\les \la x\ra^{-3-}$. Fix $\phi=Uv\psi\in S_2(L^2\times L^2)$. Then $\phi \in S_3(L^2\times L^2)$ if and only if $\psi \in L^2\times L^2$.
\end{lemma}

\begin{proof}
By  Lemma~\ref{lem:S2}, we have 
$$
\psi=-\mG_0v^*\phi.
$$
Using this and \eqref{G0 to G2 ident}, if $\phi \in S_3$, then we have
\begin{align*}
 \|\psi\|^2_2 =\la \mG_0v^*\phi,\mG_0v^*\phi\ra=2m\la v \mG_2v^*\phi,\phi\ra < \infty
\end{align*}
by the decay assumption on $v$.

Now assume that $\psi \in L^2\times L^2$. Since $\frac{x_j}{\la x\ra^2}\not \in L^2$, $j=1,2$, by Lemma~\ref{lem:moreonS1} and part ii) of Remark~\ref{rmk:ranks}, we have $w_j=0$, $j=1,2$, which implies that
\be\label{eq:yenimagic}
  \int_{\R^2}y_j I_1 v^*(y)\phi(y) dy= -\frac{i}{2m } \alpha_j M_{22}v^*\phi,  \,\,\,j=1,2.
\ee
Since $\alpha^2_j=I$, this also implies that
\be\label{eq:yenimagic1}
 \int_{\R^2} \alpha\cdot y   I_1 v^*(y)\phi(y) dy= -\frac{i}{ m }   M_{22}v^*\phi.
\ee

We are ready to prove that $S_2v\mG_1v^*S_2\phi=0$.
Recall from \eqref{def:mG1} that
$$
\mG_1=-i\alpha \cdot \nabla G_1
	+ 2mG_1 I_{1}-\frac2mM_{11}-\frac2mM_{22}.
$$
Note that the contribution of the third term is zero. We  consider the contribution of the second term. We have
$S_2 vG_1I_1v^*S_2=S_2vWI_1v^*S_2$, where $W$ is the integral operator with
kernel $-2x\cdot y$. This is because $G_1(x,y)=|x-y|^2=|x|^2-2x\cdot y+|y|^2$, and the contribution of $|x|^2+|y|^2$ is zero since $PS_2=S_2P=0$. Therefore, we have
$$
2mS_2 vG_1I_1v^*S_2\phi=-4m S_2 v(x) \int_{\R^2} (x \cdot y) I_1 v^*(y)\phi(y) dy\\
=2i S_2 v(x) (\alpha \cdot x) M_{22}v^*\phi.
$$
In the second equality we used \eqref{eq:yenimagic}.
The contribution of the first term is 
\begin{multline*}
-iS_2v \alpha \cdot \nabla G_1 v^* \phi = -2i S_2 v(x) \int_{\R^2} \alpha \cdot (x-y) v^*(y) \phi(y) dy\\
=  -2i S_2 v(x) (\alpha \cdot  x) \int_{\R^2}   v^*(y) \phi(y) dy+2i S_2 v(x) \int_{\R^2} \alpha \cdot  y  v^*(y) \phi(y) dy \\
= -2i S_2 v(x) (\alpha \cdot  x) M_{22}v^*\phi+2i S_2 v(x) I_2 \int_{\R^2} \alpha \cdot  y  v^*(y) \phi(y) dy.
\end{multline*}
In the last equality we used $PS_2=S_2P=0$.
Therefore, the sum of the contributions of the first two terms in the definition of $\mathcal G_1$ is equal to (using \eqref{eq:yenimagic1})
$$
2i S_2 v(x) I_2 \int_{\R^2} \alpha \cdot  y  v^*(y) \phi(y) dy
=2i S_2 v(x)  \int_{\R^2} \alpha \cdot  y   I_1 v^*(y) \phi(y) dy = \frac2m S_2 v(x) M_{22}v^*\phi, 
$$
which cancels the contribution of the last term in the definition of $\mathcal G_1$. This finishes the proof of the lemma. 
\end{proof}

\begin{lemma}\label{eigenproj lem}
Assume that  $v(x)\les \la x\ra^{-3-}$.  The operator 
$$P_m:=\frac1{2m} \mG_0v S_3[S_3v\mG_2v^*S_3]^{-1}S_3 v^*\mG_0$$
is the finite rank orthogonal projection  
onto the $m$ energy eigenspace of $H=D_m+V$.
\end{lemma}

\begin{proof}
Let $\{\phi_j\}_{j=1}^N$ be an orthonormal basis
for the $S_3L^2$, the range of $S_3$.  By the characterization in Lemma~\ref{lem: S1} and Remark~\ref{rmk:resonances}, the eigenspace is finite dimensional.   Then, by the lemmas above, we have
\begin{align}\label{psi j eqn}
	\phi_j=Uv\psi_j,\,\,\,\,\psi_j=-\mG_0v^*\phi_j, \qquad 1\leq j\leq N,
\end{align}
where $\psi_j\in L^2\times L^2$, $j=1,2,\ldots, N$, are eigenvectors.
Since  $\{\phi_j\}_{j=1}^N$ is linearly independent, we have that $\{\psi_j\}_{j=1}^N$ is  linearly independent, and hence it is a basis for $m$ energy eigenspace.
Using the orthonormal basis for $S_3 L^2$, we have that for any $f\in L^2\times L^2$,
$S_3f=\sum_{j=1}^N \la f,\phi_j\ra \phi_j$. Therefore, we have
\begin{align}\label{S2 sum}
	S_3v^* \mG_0f=\sum_{j=1}^N \la f,\mG_0v^*\phi_j\ra \phi_j=-\sum_{j=1}^N \la f,\psi_j\ra \phi_j.
\end{align}
This implies that the range of $P_m$ is contained in the span of $\{\psi_j\}_{j=1}^N$, since $P_m$ is self-adjoint.

We claim that, for each $i_0,j_0\in\{1,2,\ldots,N\},$
$$
\big\la \psi_{i_0}, P_m \psi_{j_0}\big\ra =\big\la \psi_{i_0}, \psi_{j_0}\big\ra.
$$
This implies that the range of $P_m$ is
equal to the span of $\{\psi_j\}_{j=1}^N$ and that $P_m$ is the identity operator on the range of $P_m$.
Since $P_m$ is self-adjoint, the assertion of the lemma holds.

Let $\mathcal A:=S_3v\mG_2v^*S_3$. Let $A=\{A_{ij}\}_{i,j=1}^N$, $B=\{B_{ij}\}_{i,j=1}^N$ be the matrix representations of $\mathcal A$ and $\mathcal A^{-1}$ with respect to the
orthonormal basis $\{\phi_j\}_{j=1}^N$ of $S_3$.
Using \eqref{G0 to G2 ident} and polarization,
\begin{align*}
	A_{ij}&=\la \phi_j,S_3v\mG_2v^*S_3\phi_i\ra=\frac{1}{2m} \la \mG_0v^*\phi_j,\mG_0v^*\phi_i\ra  = \frac{1}{2m} \la \psi_j,\psi_i\ra,\\
	B_{ij}&= A^{-1}_{ij}=\la \phi_j,\mathcal A^{-1} \phi_i\ra.
\end{align*}
Using this and \eqref{S2 sum},  we have
\begin{multline*}
\big\la \psi_{i_0}, P_m \psi_{j_0}\big\ra= \frac1{2m}\Big\la S_3v^* \mG_0  \psi_{i_0}, \mathcal A^{-1} S_3v^* \mG_0 \psi_{j_0} \Big\ra \\ =\frac1{2m}
\Big\la \sum_{i=1}^N \la \psi_{i_0},\psi_i\ra \phi_i, \mathcal A^{-1} \sum_{j=1}^N \la \psi_{j_0},\psi_j\ra \phi_j \Big\ra 
=\frac1{2m}\sum_{i,j=1}^N \la \psi_{i_0},\psi_i\ra \la  \psi_j, \psi_{j_0} \ra \big\la \phi_i, \mathcal A^{-1}   \phi_j \big\ra\\ = 2m \sum_{i,j=1}^N  A_{i,i_0} B_{j,i} A_{j_0,j} = 2m A_{j_0,i_0}=   \la \psi_{i_0},\psi_{j_0}\ra.
\end{multline*}
This finishes the proof of the claim and the lemma. 
\end{proof}

\section*{Acknowledgements} 
	The authors would like to thank Nabile Boussaid and the anonymous referee for many helpful suggestions.  In particular, suggesting relevant work on the spectrum of Dirac operators which strengthened our results.

\end{document}